\documentclass[reqno,10pt]{amsart}
\usepackage{color} 
\usepackage[left=1 in, right=1 in,top=1 in, bottom=1 in]{geometry}
\usepackage{amsfonts,amsmath,amsthm,amssymb,mathrsfs,environ}
\newtheorem{theorem}{Theorem}[section]

\newtheorem{lemma}[theorem]{Lemma}

\newtheorem{remark}[theorem]{Remark}

\usepackage[pdfstartview=FitH,colorlinks=true,backref=page]{hyperref}
\hypersetup{urlcolor=red, linkcolor=blue,citecolor=red, CJKbookmarks=true}
\allowdisplaybreaks
\makeatletter

\newcommand{\bd}{\mathrm{d}}

\newcommand{\vdot}{v\!\cdot\!\nabla_x}
\newcommand{\sdot}{\!\cdot\!}
\newcommand{\m}{\dm}
\newcommand{\bdv}{\dm \bd v}
\newcommand{\dm}{\mathrm{M}}
\newcommand{\divg}{\mathrm{div}}
\newcommand{\Q}{\mathcal{Q}}
\newcommand{\timess}{\!\times\!}
\newcommand{\curl}{\nabla\!\!\times\!\!}

\newcommand{\phs}{\mathbb{T}^3\times\mathbb{R}^3}
\newcommand{\reps}{\tfrac{1}{\epsilon}}
\newcommand{\repsa}{\tfrac{\alpha_\epsilon}{\epsilon}}
\newcommand{\repsb}{\tfrac{\beta_\epsilon}{\epsilon}}
\newcommand{\ges}{\gamma_\epsilon}
\newcommand{\repst}{\tfrac{1}{\epsilon^2}}

\newcommand{\llv}{L^2_{v,\Lambda}}

\newcommand{\dt}{\frac{\mathrm{d}}{2 \mathrm{d}t}}
\newcommand*{\rom}[1]{\expandafter\@slowromancap\romannumeral #1@}
\newcommand{\weakc}{\rightharpoonup}
\newcommand{\bdp}{\mathcal{P}}

\newcommand{\intps}{\int_{\scriptstyle \mathbb{T}^3}\!\int_{\scriptstyle \mathbb{R}^3} }

\newcommand{\intv}{\int_{\mathbb{R}^3}}
\newcommand{\intt}{\int_{\mathbb{T}^3}}

\newcommand{\tdt}{\frac{\mathrm{d}}{ \mathrm{d}t}}
\makeatother
\NewEnviron{aligno}{%
\begin{align}\begin{aligned}
  \BODY
\end{aligned}\end{align}
}

\title[ VMB]{ The diffusive limits of two species  Vlasov-Maxwell-Boltzmann equations }

\author{Xu Zhang}
\address[X. Zhang]{\newline School of Mathematics and Statistics, zhengzhou University, Zhengzhou, 450001, P. R. China}
\email{xuzhang889@zzu.edu.cn}

\begin{document}
	\begin{abstract}
 In this work, we mainly concern the limiting behavior of the electromagnetic field of two species Vlasov-Maxwell-Botlzmann system in diffusive limits. As knudsen numbers goes to zero, the electric magnetic and magnetic field may perserve or vanish.  We   verify rigorously  Navier-Stokes, Navier-Stokes-Poisson and Navier-Stokes Maxwell limit of the two species Vlasov-Maxwell-Boltzmann system on the torus in three dimension.   The justification is based on  the unified and uniform estimates of solutions to the dimensionless Vlasov-Maxwell-Boltzmann.  The uniform estimates of solutions are obtained by employing the hypocoercivity of the linear Boltzmann operator and  constructing a equation containing damping term of electric field.
	\end{abstract}
\maketitle

\section{Introduction and Motivation}

The evolution of the dilute  gas consisting of  charged particles is described by  the Vlasov-Maxwell-Boltzmann (VMB) system which models the dynamics of charged particle  under  auto-induced electromagnetic field. There exists an alternative description for the gas: the fluid dynamics. These two ways are deeply connected. From the point {of} view of physics, the fluid regime is  those with small Knudsen numbers which is defined as the ratio of the mean free path of the molecular to the physical length scale.   
As the Knudsen number goes to zero, the gas under consideration satisfies continuum mechanics. Due to its electromagnetic field, the   hydrodynamic limits  in formal level are quite involved(see \cite{diogosrm-2019-vmb-fluid,jama2012siam} for instance). While  cations and anions are with the same mass(see \cite{jama2012siam} with different mass), D.~Ars\'{e}nio and L.~Saint-Raymond  performed a very systematic formal analysis in \cite{diogosrm-2019-vmb-fluid}.  From \cite[Sec.2.4]{diogosrm-2019-vmb-fluid}, the dimensionless two species Vlasov-Maxwell-Boltzman equations are
\begin{align}
\label{vmb-two-species}
\begin{cases}
\epsilon \partial_t   f^+_\epsilon + \vdot f^+_\epsilon +  (  \alpha_\epsilon E_\epsilon +  \beta_\epsilon v\timess B_\epsilon) \sdot\nabla_v f^+_\epsilon =  \reps \Q(f^+_\epsilon,f^+_\epsilon) + \reps \Q(f^+_\epsilon,f^-_\epsilon), \\
\epsilon \partial_t   f^-_\epsilon + \vdot f^-_\epsilon - (  \alpha_\epsilon E_\epsilon +  \beta_\epsilon v\timess B_\epsilon) \sdot\nabla_v f^-_\epsilon =  \reps \Q(f^-_\epsilon,f^+_\epsilon) + \reps \Q(f^-_\epsilon,f^-_\epsilon), \\
\gamma_\epsilon \partial_t   E_\epsilon - \curl   B_\epsilon = - \tfrac{\beta_\epsilon}{\epsilon^2} \intv (f_\epsilon^+ - f_\epsilon^-) v \bd v,\\
\gamma_\epsilon \partial_t   B_\epsilon + \curl   E_\epsilon =0,\\
\divg   B_\epsilon =0,~~\epsilon\sdot \divg   E_\epsilon =\intv (f^+_\epsilon - f^-_\epsilon) \bdv.
\end{cases}
\end{align}

where
\begin{itemize}
\item $\epsilon$: the knudsen number;
\item $\alpha_\epsilon$:  strength of the electric induction;
\item $\beta_\epsilon$:  strength of the magnetic induction;
\item $\gamma_\epsilon$: ratio of the bulk velocity to the speed of the light;
\item $E_\epsilon$: eletric field;
\item $B_\epsilon$: magnetic field;
\item $f_\epsilon^\pm$: number density of anion or cation.
\end{itemize}
Specially, $\alpha_\epsilon$, $\beta_\epsilon$ and $\gamma_\epsilon$ satisfy the following relations:
\begin{align}
\label{relationabg}
\alpha_\epsilon \gamma_\epsilon = \epsilon\sdot \beta_\epsilon. 
\end{align}

$f_\epsilon^\pm=f_\epsilon^\pm(t,x,v)$   denotes the density of the charges with velocity $v$ ($v\in\mathbb{R}^3$)  on position $x$ ($x\in \mathbb{T}^3$) at time $t$ ($t>0$). The charged particles are  moving under the Lorentz force (reprensented by the third term on the right hand of the first equation in \eqref{vmb-two-species}) and the inter-particle collisions ( modeled by $\Q(f,f)$ ) along their trajectories. The self-generated electromagnetic field is described by the Maxwell's equation (the equations from the second to the last in \eqref{vmb-two-species}). In sequence, they are   the Ampère's equation (the second, $E$ is the electric field),  Faraday's equation (the third, $B$ denotes the magnetic field) and the Gauss's law (the last).  $\divg B=0$ means that there is no magnetic monopole. For the two species equations, the notations can be understood in the similar way. The differences are that directions of the Lorentz force acting on ions and electrons are opposite.

The collission operator $\Q$  is  defined as follows:
\begin{equation}\label{collision-original}
\begin{split}
\Q(f,f)=&\int_{\mathbb{R}^3\times\mathbb{S}^{2}}(f'f'_*-ff_*) b(v-v_*,\omega) \bd v_*\bd\omega\,,
\end{split}
\end{equation}
where $f'=f(v')$, $f'_*= f(v'_*)$, $f_*= f(v_*)$.  Here, $v$ and $v_*$ denote the velocities of two particle before the   collision and $v'$ and $v_*'$ are their velocities after collision. The collision between particles  is elastic, i.e., satisfies the conservation law of momentum and energy:
\begin{align*}
\begin{cases}
v+ v_* = v' + v'_*, \\
|v|^2 + |v_*|^2 =    |v'|^2 + |v'_*|^2\, , 
\end{cases}
\end{align*}
and
\begin{align*}
\begin{cases}
v' = \frac{v + v_*}{2} + \frac{|v-v_*|}{2}\omega \\
v_*'= \frac{v + v_*}{2} - \frac{|v-v_*|}{2}\omega\,, ~~\omega =\frac{v' -  v_*'}{| v' - v_*'|} \in \mathbb{S}^2 ~~(\text{the unit sphere in } \mathbb{R}^3).
\end{cases}
\end{align*}

The nonnegative   $b(v-v_*, \omega)$, called {  cross-section}, is a function of    $|v-v_*|$  and cosine of the derivation angle $ (\frac{v-v_*}{|v-v_*|}, \omega)$.

For two species cases, it is more convenient to consider   sum and difference of $f_\epsilon^+$ and $f_\epsilon^-$. Denoting
\[ F_\epsilon = f^+_\epsilon + f^-_\epsilon, G_\epsilon= f^+_\epsilon - f^-_\epsilon,  \]
then we have
\begin{align}
\label{vmb-two-species-+-}
\begin{cases}
\epsilon\partial_t F_\epsilon + \vdot F_\epsilon +  ( \alpha_\epsilon E_\epsilon +  \beta_\epsilon v\timess B_\epsilon) \sdot\nabla_v G_\epsilon = \reps \Q(F_\epsilon,F_\epsilon) \\
\epsilon \partial_t G_\epsilon + \vdot G_\epsilon +  ( \alpha_\epsilon E_\epsilon + \beta_\epsilon v\timess B_\epsilon) \sdot\nabla_v F_\epsilon = \reps \Q(G_\epsilon,F_\epsilon)  \\
\gamma_\epsilon\partial_t E_\epsilon - \curl B_\epsilon = -  \tfrac{\beta_\epsilon}{\epsilon^2} \intv G_\epsilon v \bd x, \\
\gamma_\epsilon \partial_t B_\epsilon + \curl E_\epsilon =0,\\
\divg B_\epsilon =0,~~\divg E_\epsilon = \tfrac{\alpha_\epsilon}{\epsilon^2} \intv G_\epsilon \bd v.
\end{cases}
\end{align}
Plugging the following ansatz into \eqref{vmb-two-species-+-}
\[ F_\epsilon = \m ( 1 + \epsilon f_\epsilon ),~~ G_\epsilon = \epsilon \m g_\epsilon, \]
it follows that
\begin{align}
\label{vmbtwolinear}
\begin{cases}
\partial_t f_\epsilon +  \reps \vdot f_\epsilon  +   \repst \mathcal{L}(f_\epsilon)  = N_f, \\
\partial_t g_\epsilon +  \reps \vdot g_\epsilon   + \tfrac{\alpha_\epsilon}{\epsilon^2}  E_\epsilon \sdot v +   \repst \mathsf{L}(g_\epsilon)  = N_g, \\
\ges\partial_t E_\epsilon - \curl B_\epsilon = - \tfrac{\beta_\epsilon}{\epsilon} j_\epsilon, \\
\ges  \partial_t B_\epsilon + \curl E_\epsilon =0,\\
\divg B_\epsilon =0,~~\divg E_\epsilon = \tfrac{\alpha_\epsilon}{\epsilon} n_\epsilon,
\end{cases}
\end{align}
with
\begin{align*}
  j_\epsilon & = \intv g_\epsilon v \bdv,~n_\epsilon  = \intv g_\epsilon    \bdv,\\
 N_f & =  \tfrac{\alpha_\epsilon}{\epsilon} E_\epsilon\sdot v \sdot g_\epsilon -  ( \repsa E_\epsilon + \repsb v \times B_\epsilon) \sdot \nabla_v g_\epsilon + \reps \Gamma(f_\epsilon, f_\epsilon), \\
  N_g & =  \tfrac{\alpha_\epsilon}{\epsilon} E_\epsilon\sdot v \sdot f_\epsilon - (\repsa E_\epsilon +  \repsb v \times B_\epsilon) \sdot \nabla_v f_\epsilon + \reps \Gamma(g_\epsilon, f_\epsilon),\\
- \mathsf{L}(w)&  = (\m)^{-1}\Q(\m w, \m), ~-\mathcal{L}(w)  = (\m)^{-1}\left( \Q(\m w, \m) + \Q(\m, \m  w)\right).
\end{align*}

In this work, we concern the hydrodynamics limit, i.e.,  the transition from kinetic system to macroscopic fluid equations. The mathematical justification of the hydrodynamics limit for the Boltzmann type equations can be   divided into two sorts: the renormalized solution framework or classic solution framework. For the  work on the existence of renormalized solutions and fluid limit in renormalized solutions work, one can check \cite{bgl1993convergence,diperna-lions1989cauchy,gsrm2004,lm2010soft,lm2001acoustic,masmoudi-srm2003stokesfourier,mischler2010asens}. In this work, we only concern the classic solution framework. The existence of classic solution to VMB system can be found in \cite{dlyz2017cmp,guo-2003-vmb-invention,sr2006-vmb} and the references therein.  The rigorous verification work started from the Boltzmann equation and then was generalized to Vlasov-Poisson-Boltzmann or Vlasov-Maxwell-Boltzmann  system. For the Boltzmann equation, the limiting system is just Navier-Stokes system under diffusive limit. The justification work could  be found in \cite{bu1991,briant-2015-be-to-ns,guo2006NSlimit,ns-limit-2018} and the references therein. For the VPB system, according to \cite{diogosrm-2019-vmb-fluid}, the limiting fluids equation could be Navier-Stokes equation or Navier-Stokes-Poisson equation. For the diffusive limit of the Valsov-Poisson-Boltzmann equation, we refer to \cite{uvpb2020,jz2020vpbconvergence,vpb2020limit-spectrum,twovpblimits} and the references therein for the rigorous justification work.  

Due to the presence of the electromagnetic field, the structure of limiting fluid equations of VMB system is richer.  The electric or magnetic field may vanish as the Knudsen number tends to zero.  The Navier-Stokes-Poisson limit  of VMB system was verified in \cite{vmbtonsp} by Hilbert expansion method. The Navier-Stokes-Maxwell  limit   can be found in \cite{vmbtonswh} by Hilbert expansion method and in \cite{vmbtonswu} based on the uniform estimates with respect to the knudsen number. The authors in \cite{vmbtonsp,vmbtonswu,vmbtonswh} all consider the diffusive limit of the VMB system, but the electromagnetic field in the limiting equation are different. There only exists electric field in \cite{vmbtonsp}. Both electric field and magnetic field preserve in \cite{vmbtonswu,vmbtonswh}. The limit behavior of the electromagnetic field in VMB system is determined by the scalings  of the strength of electric induction $\alpha_\epsilon$ and magnetic induction $\beta_\epsilon$.  For details, denoting 
 \[ \alpha = \lim\limits_{\epsilon \to 0} \repsa,~~\beta = \lim\limits_{\epsilon \to 0} \beta_\epsilon,  \]
 then the type of limiting fluid equation are as follows
 \begin{aligno}
 \label{transisiter}
 &(1)~ \alpha >0,~\beta>0, ~\to \text{Electric field and magnetic field}~(see~ \eqref{nsw});\\
 &(2)~ \alpha >0,~\beta=0, ~\to \text{only electric field}~(see~ \eqref{nsp});\\
 &(3)~ \alpha =0,~\beta=0, ~\to \text{No electromagnetic field} ~(see~ \eqref{nsf}).
 \end{aligno}

 The goal of  this work is  to verify rigorously \eqref{transisiter}. We try to justify the role of $\alpha_\epsilon$ and $\beta_\epsilon$ in determining the limiting system. In other word, we shall verify the diffusive limit of two species Vlasov-Maxwell-Boltzmann for the three cases in \eqref{transisiter}.  Since there exist lots of cases for each one in \eqref{transisiter}, for simplicity,  we choose the following three main scalings  for example:
 \begin{aligno}
 \label{alphabg}
 &(1)~\alpha_\epsilon = \epsilon,~~\beta_\epsilon =1, ~\to \text{Navier-Stokes-Maxwell system};\\
 &(2)~ \alpha_\epsilon = \epsilon,~~\beta_\epsilon =\epsilon, ~\to \text{Navier-Stokes-Poisson system};\\
 &(3)~ \alpha_\epsilon = \epsilon^2,~~\beta_\epsilon =\epsilon^2, ~\to \text{Navier-Stokes system}.
 \end{aligno} 
Of course, our method can be generalized to more scalings (see Remark \ref{remark-scalings}). The justification of this transition phenomenon is based on  the uniform estimates of solutions $(f_\epsilon, g_\epsilon, B_\epsilon, E_\epsilon)$ with respect to the knudsen number. Since the local coercivity properties of the linear Boltzmann operator, the difficulty of obtaining the uniform estimates is how to get the dissipative energy of solutions, special for the macroscopic parts.  

The approach of this work is motivated by the method used in   \cite{briant-2015-be-to-ns,mouhotneumann-2006-decay} where the authors added a ``mixed term'' up to the instant energy to deduce the hypocoercivity  properties of the Boltzmann equations. By employing the similar strategy, we can get the dissipative energy(see \eqref{dissi} for the meanings) of $f_\epsilon$ and $g_\epsilon$.  Since the equations of $B_\epsilon$ and $E_\epsilon$ are hyperbolic, the key difficulty is how to obtain the dissipative estimates  of the eletromagnetic field for all the three cases in \eqref{alphabg} at the same time.   The strategy of dealing with this difficulty is to construct a equation with damping term of $E_\epsilon$. In details, multiplying the equation of $g_\epsilon$ by $\tilde{v}$(see \eqref{estab}) and then integrating over $\mathbb{R}^3$, we can obtain that
\begin{align}
\label{idea}
-   \partial_t \tilde{j}_\epsilon + \cdots + \sigma\tfrac{\alpha_\epsilon }{\epsilon^2} E_\epsilon = \cdots.
\end{align}
There exists kind of damping term   $E_\epsilon$ in the above equation. Based on ``mixture norm'', we can obtain the dissipative estimates of the curl part of $E_\epsilon$ first. After that, we can obtain the dissipative estimates of the curl part of $B_\epsilon$ by employing the Ampère and Faraday's equation. Then we can recover the dissipative estimates of electromagnetic field by the Helmholtz decomposition.  As for the three difference cases in \eqref{alphabg}, the first priority is to obtain the dissipative estimates of electromagnetic field related to $\alpha_\epsilon$ and $\beta_\epsilon$ (see the coefficient $\tfrac{\alpha_\epsilon^2}{\epsilon^2}$ in \eqref{esttheowhole}). After determining the coefficient $\tfrac{\alpha_\epsilon^2}{\epsilon^2}$, the inequality can be closed  with the help of \eqref{relationabg} which is very natural and play a essential role in obtaining the uniform estimates.

The novelty of this work is twofold.  One is that we  verify the diffusive limits for the three main scalings \eqref{alphabg} at the same time. We also give a rigorous proof for the formal derivation in \cite{diogosrm-2019-vmb-fluid} (only for two species with strong interaction).   On the other hand, the method of this work is based on the uniform estimates obtained by                                ``mixed norm'' introduced in \cite{mouhotneumann-2006-decay,briant-2015-be-to-ns}. Specially,   as mentioned before, we employ a new way of obtaining the dissipative estimates of electromagnetic field.  Indeed,  for instance(\cite{vmbtonswu}), noticing that the magnetic field is divergence free, by  employing the third and forth equation in \eqref{vmbtwolinear}, one can obtain the dissipative energy estimates by employing the wave equation of $B_\epsilon$. Besides, based on the micro-macro decomposition, the norm used in \cite{guo-2003-vmb-invention} must contain the temporal derivative to obtain the dissipative estimate of electromagnetic field( see \cite[Lemma 9]{guo-2003-vmb-invention} for more details). We also give an alternative proof for the existence of classic solutions to the 	VMB system and the three limiting fluid systems.

The rest of this paper is made up of three sections. We shall introduce the preliminaries  in the Section \ref{sec-settings-main}, such as notations, assumption on kernel and initial data. The main results can be found in Sec \ref{sec-results} where we also explain our strategy and difficulties. The Section \ref{secEstimates} consists in deriving the uniform a prior  estimates and is the most important part of this work. Based on the uniform estimates for all the three cases in \eqref{alphabg}, the fluid limits will be verified in Sec. \ref{sec-limit}.

\section{Preliminaries}
\label{sec-settings-main}

\subsection{Notations and Terms}
For estimates like this
\begin{align}
\label{dissi}
\tdt E(t)  + D(t) \le \cdots,
\end{align}  
where $E(t)$ and $D(t)$ are positive functions of $t$. We call $E(t)$ by the instant energy and $D(t)$ the dissipative energy (estimates). $\nabla^i_x  f = \nabla_{x_1}^{i_1} \nabla_{x_2}^{i_2} \nabla_{x_3}^{i_3} f$  is   the $i$-$th$ derivative  of $f$ with $i_1 + i_2 + i_3 = i$. Specially, $\nabla_x f$ is the gradient of scalar function $f$.   In the similar way, we can define $\nabla_v^i f$ and $\nabla_v f$.  The norms of $f$ are defined as follows:
\begin{align*}
&\|f\|_{L^2_v}^2 =  \intv f^2 \bdv ,   ~~\|f\|_{L^2}^2 =  \intps f^2 \bdv \bd x,\\
& \|f\|_{H^s_x}^2 =  \sum\limits_{k=0}^s\|\nabla^k_x f\|_{L^2}^2,  \|f\|_{H^s}^2 =  \sum\limits_{k=0}^s \sum\limits_{i + j =k}\|\nabla^i_x \nabla^j_v f\|_{L^2}^2.
\end{align*} 
Denoting $\hat{v}^2= 1 + |v|$,   the norms with weight on $v$ are defined as follows: 
\begin{align*}
& \|f\|_{L^2_\Lambda}^2 = \|f\hat{v}\|_{L^2}^2,~ \|f\|_{H^s_{\Lambda_x}}^2 =  \sum\limits_{k=0}^s\|\nabla^k_x f {\hat{v}}\|_{L^2}^2,\\
& \|f\|_{\llv}^2 =  \intv f^2 \hat{v}^2\bdv
,~~\|f\|_{H^s_\Lambda}^2 =  \sum\limits_{k=0}^s \sum\limits_{i + j =k}\|\nabla^i_x \nabla^j_v f\|_{L^2_\Lambda}^2.
\end{align*} 
In this work, the $C$ denotes some constant independent of $\epsilon$ and is different from lines to lines. 
The $a \lesssim b$ means that there exists some constant $C$ independent of $\epsilon$ such that $a \le C b$.

\subsection{Limiting equations}
\label{subsec-ns}

For the convenience of stating our main results, we list three fluid equations.
The first one is the Navier-Stokes-Maxwell system where exists Lorentz force(both electric and magnetic field):
\begin{align}
\label{nsw}
\text{(NSW)}~
\begin{cases}
\partial_t u + u\sdot \nabla u - \nu \Delta u + \nabla P =   n \sdot E +     j \timess B,\\
\partial_t \theta + u \sdot\nabla \theta - \kappa \Delta \theta =0,\\
\divg u = 0,~~ \rho + \theta =0,\\
\partial_t E - \curl B = -   j,\\
\partial_t B + \curl E =0,\\
\divg B =0, ~~\divg E =  n,\\
j= \sigma (   E  +  u\timess B) -  \sigma\nabla_x n + n u.
\end{cases}
\end{align}
While there is only Coulomb force, we obtain the Naiver-Stokes-Poisson system: 
\begin{align}
\label{nsp}
(\text{NSP})
\begin{cases}
\partial_t u + u\sdot \nabla u - \nu \Delta u + \nabla P =   n \sdot E,\\
\partial_t \theta + u \sdot\nabla \theta - \kappa \Delta \theta =0,\\
\partial_t n + \divg(nu) -  {\sigma} \Delta n +  {\sigma} n =0,\\
\divg u = 0,~~ \rho + \theta =0,\\
\curl E =0, ~~\divg E =  n.
\end{cases}
\end{align}
The last case is two fluid Navier-Stokes system where the electromagnetic field vanishes:
\begin{align}
\label{nsf}
(\text{NSF})
\begin{cases}
\partial_t u + u\sdot \nabla u - \nu \Delta u + \nabla P = 0,\\
\partial_t \theta + u \sdot\nabla \theta - \kappa \Delta \theta =0,\\
\partial_t n + u\sdot\nabla n - \sigma\Delta n  =0,\\
\divg u = 0,~~ \rho + \theta =0.
\end{cases}
\end{align}
In the above system, the $\nu,~\kappa,~\sigma$ are strictly positive constants and will be clearly defined in \eqref{nusigma}.

\subsection{Assumption on the linear operators}
\label{sec-assump-on-L}
This section consists in stating the assumption on the linear Boltzmann operators: $\mathcal{L}$ and $\mathsf{L}$. According to our notations,  by simple computation, we can infer
\begin{align*}
\mathsf{L}(h)& =  (\m)^{-1}\Q(\m h, \m) \\
& = \int_{\mathbb{R}^3\times\mathbb{S}^{2}}(h'-h)\m(v_*) b(v-v_*,\omega) \bd v_*\bd\omega\\
& = -\int_{\mathbb{R}^3\times\mathbb{S}^{2}}h'\m(v_*) b(v-v_*,\omega) \bd v_*\bd\omega + \int_{\mathbb{R}^3\times\mathbb{S}^{2}}\m(v_*) b(v-v_*,\omega) \bd v_*\bd\omega \sdot h \\
&:=-\mathbf{\Phi}(h) + \Lambda(v) \sdot h,
\end{align*}
and
\begin{align*}
\mathcal{L}(h) & = (\m)^{-1}\Q(\m h, \m)  + (\m)^{-1}\Q(\m , \m h)\\
& = \mathbf{\Phi}(h) - \Lambda(v) \sdot h + (\m)^{-1}\Q(\m , \m h)\\
&:= -\mathbf{K}(h) + \Lambda(v) \sdot h.
\end{align*} 
In summary, the operators satisfy
\begin{align} 
\label{assump-h1-h}
\mathcal{L} = -\mathbf{K} + \Lambda,~~ \mathsf{L} = -\mathbf{\Phi} + \Lambda.
\end{align}
From \cite{gsrm-2005-survery}, the operator $\Lambda(v)$ is coercive. Even though the compact parts of $\mathcal{L}$ and $\mathsf{L}$ are different, according to \cite{belm-2000-binary,guo-2003-vmb-invention,mouhot-2006-cpde,mouhot-2006-homogeneous},  the similar assumption can be established. To avoid using too many notations of constants and without loss of generality, we set the constant to be one.

\noindent {\bf H1.}   This assumption is on the coercive operator $\Lambda$ in \eqref{assump-h1-h}.  
\begin{itemize}
\item  {For any $h, g \in L^2_v$,  
\[  \|h\|_{\llv}^2 \le \intv \Lambda(h)\sdot h \bdv   \le C \|h\|_{\llv}^2, \] 
and}
\begin{align}
\vert\intv  \Lambda(h)\sdot g\bdv  \vert \le C \|h\|_{\llv}\|g\|_{\llv}.
\end{align}
\item With respect to the derivative of $v$, the operator $\Lambda$ admits ``a defect of coercivity'', i.e.,
there exist some strictly  positive  constant  $\delta$ and $C_\delta$   such that
\begin{aligno}
\label{constant-a3-a4}
\intv\nabla_{v} \Lambda(h)\sdot\nabla_{v} h\bdv   \geqslant ( 1 -\delta) \left\|\nabla_{v} h\right\|_{\llv}^{2}-C_\delta\|h\|_{L^{2}_v}^{2}, ~~0< \delta <1.
\end{aligno}
and for the higher order derivative, 
\begin{aligno}
\label{constant-a3-a4-hi}
\intps \nabla_{v}^i \nabla_x^j \Lambda(h)\sdot\nabla_{v}^i \nabla_x^j h\bdv \bd x \geqslant (1-\delta) \left\|\nabla_{v}^i \nabla_x^j h \right\|_{L^2_\Lambda}^{2}-C_\delta\|h\|_{H^{i+j-1}}^{2}.
\end{aligno}
\end{itemize}

\noindent {\bf H2.} This assumption is about the compact operators in \eqref{assump-h1-h}.  For any $\delta>0$, there exists some positive constant $C_\delta$ such that
\begin{aligno}
\label{constant-delta}
& \vert \intv \nabla_{v} \mathbf{K}(h) \sdot \nabla_{v} h\bdv  \vert +\vert \intv \nabla_{v} \mathbf{\Phi}(h) \sdot \nabla_{v} h\bdv  \vert  \leqslant C_\delta\|h\|_{L^{2}_v}^{2}+\delta\left\|\nabla_{v} h\right\|_{\llv}^{2},
\end{aligno}
and for higher order derivative,
\begin{aligno}
&\vert \intps \nabla_{v}^i \nabla^j_x  \mathbf{K}(h) \sdot \nabla_{v}^i \nabla^j_x h\bdv  \vert+\vert \intps \nabla_{v}^i \nabla^j_x  \mathbf{\Phi}(h) \sdot \nabla_{v}^i \nabla^j_x h\bdv  \vert \\
&\quad  \leqslant C_\delta\|h\|_{H^{i+j-1}}^{2}+\delta\left\|\nabla_{v}^i \nabla^j_xh\right\|_{L^2_\Lambda}^{2}.
\end{aligno}

\noindent {\bf H3.} (Relaxation to the local equilibrium.) This assumption is on the linear Boltzmann operators and their kernel spaces.
The linear Boltzmann operators $\mathcal{L}$  and   $\mathsf{L}$ are closed and self-adjoint operators from $L^2_v$ to $L^2_v$. 
 The kernel space of $\mathcal{L}$ is spanned by $1, v_1,v_2,v_3, |v|^2$.  The kernel space of $\mathsf{L}$ is spanned by $1$. 
Furthermore, $\mathcal{L}$ and $\mathsf{L}$ satisfy ``local coercivity assumption'': there exists $a_2>0$ such that
\begin{align}
\label{a2}
\begin{split}
&  \intv \mathcal{L}(g) \sdot g \bdv \ge  \|g - \bdp g\|_{\llv}^2, \\
&   \intv \mathsf{L}(h) \sdot h \bdv \ge  \|h - \bdp h\|_{\llv}^2, 
\end{split}
\end{align} 
where $\bdp$ is the projection operator of $\mathcal{L}$ and $\mathsf{L}$ onto their kernel space respectively, the macroscopic part. 
{In this work, while the projection $\bdp$ is related to $\mathcal{L}$, i.e., the first one in \eqref{a2}, 
\begin{aligno}
\bdp g = \intv g \bdv + v\sdot \intv g v \bdv + \tfrac{|v|^2 -3}{2} \intv g \tfrac{|v|^2 -3}{3} \bdv.
\end{aligno} 
If the projection $\bdp$ is related to $\mathsf{L}$, i.e., the second one in \eqref{a2}, 
\begin{aligno}
\bdp h = \intv h \bdv.
\end{aligno} 
}
In this work,  we use the same $\bdp$ for the projection operator of $\mathcal{L}$ and $\mathsf{L}$ onto their kernel space. 

Furthermore, we also assume that  
\begin{align}
\vert \intv f \sdot \mathcal{L}(g) \bdv \vert  \le C \|f\|_{\llv} \|g\|_{\llv},~~\vert \intv f \sdot \mathsf{L}(g) \bdv \vert  \le C \|f\|_{\llv} \|g\|_{\llv},~~\forall f,~g \in L^2_{\llv}.
\end{align}

\noindent {\bf H4.} This assumption is on $\Gamma(g,g)$ and $\Gamma(g,h)$. 
\begin{itemize}
\item  {For any $g, h\in L^2$,  $\Gamma(g,g) \in \mathrm{Ker}(\mathcal{L})^\perp$,~~ $\Gamma(g,h) \in \mathrm{Ker}(\mathsf{L})^\perp$.}
\item  {For the non-linear operator and $s \ge 3$  }
\begin{align}
\label{constant-cn}
\begin{split}
&\vert \intps \nabla^s_x \Gamma(g,h) \sdot f \bdv \bd x \vert \lesssim \|(g,h)\|_{H^s_{x}} \|(g,h)\|_{H^s_{\Lambda_x}}\|f^\perp\|_{L^2_\Lambda},\\
&\vert \intps \nabla^j_x \nabla_v^i \Gamma(g,h) \sdot f \bdv \bd x \vert \lesssim \|(g,h)\|_{H^s} \|(g,h)\|_{H^s_{\Lambda }}\|f\|_{L^2_\Lambda},~~i\ge 1,~~ s = i+j.
\end{split}
\end{align}
\end{itemize}
Besides, we introduce matrices $A(v)$ and $B(v)$, vector $\tilde{v}$ as follows
\begin{aligno}
\label{estab}
A(v) = v \otimes v - \tfrac{|v|^2}{3}\mathbb{I},~~B(v) = v ( \tfrac{|v|^2}{2} - \tfrac{5}{2}), ~\mathcal{L}\hat{A}(v) = A(v),~~\mathcal{L}\hat{B}(v) = B(v),~~\mathsf{L}\tilde{v} = v.
\end{aligno}
and define
\begin{align}
\label{nusigma}
 \nu = \tfrac{1}{15} \sum\limits_{ 1 \le i \le 3 \atop 1 \le j \le 3}\intv A_{ij}\hat{A}_{ij}\bdv,~~\kappa = \tfrac{2}{15} \sum\limits_{ 1 \le i \le 3  }\intv B_{i}\hat{B}_{i}\bdv,~~\sigma = \tfrac{1}{3}\intv \tilde{v}\sdot v \bdv.
\end{align} 
The matrix \eqref{estab} and constants \eqref{nusigma} will be very useful during the justification of fluid limits.

\subsection{Assumption on the initial data}
Denoting 
\begin{align*}
&\rho_\epsilon = \intv f_\epsilon \bdv,~ u_\epsilon = \intv f_\epsilon v\bdv,~\theta_\epsilon = \intv \left( \tfrac{|v|^2}{3} -1\right) f_\epsilon \bdv,~n_\epsilon = \intv g_\epsilon \bdv,
\end{align*}
similar to \cite{briant-2015-be-to-ns,guo-2003-vmb-invention}, by assuming the initial data have the same mass, velocity and total energy as the steady case, we can assume the initial data 
\begin{align}
\label{estMeaninitial}
\begin{split}
\intt \left(u_\epsilon + \gamma_\epsilon E_\epsilon \times B_\epsilon\right)(0) \bd x = 0,\\
  \intt \left(\theta_\epsilon +   \epsilon \sdot \tfrac{|E_\epsilon|^2 + |B_\epsilon|^2}{3} \right)(0) \bd x  =0,\\
    \intt \rho_\epsilon(0) \bd x = \intt n_\epsilon(0) \bd x =0,~~\intt B_\epsilon(0) \bd x= 0.
\end{split}
\end{align}

\section{ Main results}
\label{sec-results}

We  recall the three main scalings in \eqref{alphabg}. 
\begin{align}
\label{relationnsw}
\alpha_\epsilon = \epsilon, ~~ \gamma_\epsilon = \beta_\epsilon =1;\\
\label{relationnsp}
\alpha_\epsilon = \gamma_\epsilon = \beta_\epsilon =\epsilon;\\
\label{relationnsf}
\alpha_\epsilon =  \epsilon^2,~~\gamma_\epsilon = \epsilon,~~ \beta_\epsilon =\epsilon^2.
\end{align}

Define
\begin{align*}
\mathcal{H}_\epsilon^s(t) := \|(f_\epsilon, g_\epsilon, B_\epsilon, E_\epsilon)\|_{H^s_x}^2 + \epsilon^2\|(\nabla_v f_\epsilon, \nabla_v g_\epsilon)\|_{H^{s-1}}^2.
\end{align*}

\begin{theorem}
\label{theoremexi}
Under the assumption  in the section \ref{sec-assump-on-L} and the assumption \eqref{estMeaninitial} on the initial data,  for all the three scalings  (\ref{relationnsw}-\ref{relationnsf}) on $\alpha_\epsilon,~\beta_\epsilon$ and $\gamma_\epsilon$,   there exists some small enough constant $c_0$  such that for any $ 0 < \epsilon \le 1$ if  the initial data $(f_\epsilon(0), g_\epsilon(0), E_\epsilon(0), B_\epsilon(0))$ satisfy  
\[{H}_\epsilon^s(0) \le c_0,~~s \ge 3,  \]
then system \eqref{vmbtwolinear} admit a unique global classic solution $(f_\epsilon, g_\epsilon, B_\epsilon, E_\epsilon)$  satisfying for any $t>0$
\begin{align}
\label{esttheowhole}
\sup\limits_{ 0 \le s \le t}  {H}_\epsilon^s(t) +  \tfrac{1}{4} \int_0^t \left(   \|(   f_\epsilon,    g_\epsilon)\|_{H^{s}_\Lambda}^2 +  \tfrac{\alpha_\epsilon^2}{\epsilon^2} \|(E_\epsilon, B_\epsilon)\|_{H^{s-1}_{ x}}^2 +  \repst \|(   f_\epsilon^\perp,    g_\epsilon^\perp)\|_{H^{s}_{\Lambda_x}}^2 \right)(s)\bd s \le \tfrac {c_u}{c_l}  {H}_\epsilon^s(0),
\end{align}
where $c_l$ and $c_u$ are positive constants only dependent of the Sobolev embedding constant. 
\end{theorem}

\begin{remark}
We  use a equivalent norm $\tilde{H}_\epsilon^s$ (\eqref{normtildeH}) instead of $H^s_\epsilon$ to obtain the prior estimate \eqref{esttheowhole}. The constants $c_l$ and $c_u$ come from the equivalent relation of $\tilde{H}_\epsilon^s$ and $H^s_\epsilon$, see \eqref{estclcu}.

\end{remark}

\begin{remark}
We comment here on the dissipative estimates of $B_\epsilon$ and $E_\epsilon$. On one hand, while initial data belongs to $H^s$ space, we only obtain the order $s\!-\!1$ dissipative estimates of $E_\epsilon$ and $B_\epsilon$. This can be understood like this: from \eqref{idea} and \eqref{est-tilde-v},
\[-  \partial_t \tilde{j}_\epsilon + \sigma\tfrac{\alpha_\epsilon}{\epsilon^2} E_\epsilon -\reps \divg  \intv \tilde{v}\otimes v g_\epsilon  \bdv   = \cdots,\]
there already exists one derivative for the third term on the left hand of the above equation.

Furthermore,   by the Helmholtz decomposition, we can split $E_\epsilon$ and $B_\epsilon$ to two parts: curl-free part and divergence-free part. The loss derivative only occurs in the process of obtaining    $\curl E_\epsilon$ and $\curl B_\epsilon$. Indeed, 
\[ \divg E_\epsilon = n_\epsilon, \divg B_\epsilon =0.  \]
For the Vlasov-Poisson-Boltzmann system, since the eletromagnetic field is gradient flow, there is no loss of derivative.
\end{remark}

\begin{remark}
There exists a coefficient $\frac{\alpha^2_\epsilon}{\epsilon^2}$ before the dissipative estimates of the electromagnetic field. Under the scalings  \eqref{relationnsf}, i.e., $\alpha_\epsilon = o(\epsilon)$, the dissipative estimates of $E_\epsilon$ is very weak. Due to the loss of derivative and the coefficient $\tfrac{\alpha_\epsilon^2}{\epsilon^2}$, we need to adjust the coefficent and derivative carefully to close the inequality.
\end{remark}

Before  stating the limiting, let $u_0, \theta_0, n_0, E_0, B_0 \in H^s_x$ and satisfy (up to a subsequence)
\[ \mathbf{P} u_\epsilon(0) \to u_0,~~\tfrac{3}{5}\theta_\epsilon(0) - \tfrac{2}{5}\rho_\epsilon(0)  \to \theta_0, n_\epsilon(0) \to n_0,~E_\epsilon(0) \to E_0,~B_\epsilon(0) \to B_0, \text{in}~~H^{s-1}_x.\]
where $\mathbf{P}$ is the Leray projector. 
\begin{theorem}[Fluid limit]
\label{theoremlimit}
Under the assumption  in the section \ref{sec-assump-on-L} and the assumption \eqref{estMeaninitial} on the initial data, for  the solutions $f_\epsilon,~g_\epsilon,~E_\epsilon,~B_\epsilon$ constructed in Theorem \ref{theoremexi}, it follows that  
\begin{aligno}
\label{esttheolimit}
f_\epsilon \to \rho + u \sdot v + \tfrac{|v|^2-3}{2} \theta, ~~ g_\epsilon \to n(t,x),~~\text{in}~~~L^2(0,+\infty);H^{s-1}_x),\\
E_\epsilon \to E,~~B_\epsilon \to B,~~\text{in}~~~L^2((0,T);H^{s-1}_x))( \text{for any}~~  T>0),
\end{aligno} 
with $\rho,~u,~\theta,~E,~B$ satisfying 
\begin{aligno}
\label{functionspace}
\rho,~u,~\theta,~n,~E,~B \in L^\infty((0,\infty);H^s_x).
\end{aligno}
and
\begin{itemize}
\item for case \eqref{relationnsw}, $\rho,~n,~ u,~ \theta,~E$ and $B$  are global classic solutions to  Navier-Stokes-Maxwell system with initial data $(u_0,\theta_0,B_0,E_0)$;
\item  for case \eqref{relationnsp}, $\rho,~u,~\theta,~n$ and $E$ are global classic solutions to Naiver-Stokes-Poisson system  with initial data $(u_0,\theta_0,n_0)$;
\item for case \eqref{relationnsf}, $\rho,~u,~\theta$ and $n$ are global classic solutions to two fluids Navier-Stokes-Fourier system with initial data $(u_0,\theta_0,n_0)$.
\end{itemize} 
Furthermore, for any $\tau>0$, we can infer
\begin{align}
\label{well}
\mathbf{P} u_\epsilon \to u,~~\tfrac{3}{5}\theta_\epsilon - \tfrac{2}{5}\rho_\epsilon \to \theta,~~\text{in},~~C([\tau, +\infty); H^{s-1}_x).
\end{align}

\end{theorem}

\begin{remark}
\label{remark-scalings}
The main result of Theorem \ref{theoremexi} and Theorem \ref{theoremlimit} also work for more general scalings.   \eqref{relationnsp} can be generalized to 
\[ \alpha_\epsilon = \epsilon,~~   \beta_\epsilon =o(1).  \]
\eqref{relationnsw} can be generalized to 
\[ \alpha_\epsilon = o(\epsilon),~~ \beta_\epsilon = o(1),~~\gamma_\epsilon = o(1),~~\gamma_\epsilon  \lesssim \repsa,  \]
where
\begin{align*}
o(1)=\{ z_\epsilon\vert z_\epsilon \to 0,~~\epsilon \to 0  \},~~o(\epsilon)=\{ z_\epsilon\vert \tfrac{z_\epsilon}{\epsilon} \to 0,~~\epsilon \to 0  \}.
\end{align*}
\end{remark}

\begin{remark}
During the justification of the fluid limit, the key point is to verify the Ohm’s law
\[\reps j_\epsilon \to j= \sigma (  \alpha E  + \beta u\timess B) -  \sigma\nabla_x n + n u,~~\text{in the distributional sense}.\]
\end{remark}

\begin{remark} Let $\tilde{u}_0, \tilde{\rho}_0$ and $\tilde{\theta}_0$ be the limits of $u_\epsilon(0),~\rho_\epsilon(0)$ and  $ \tilde{\theta}_\epsilon(0)$ in the distributional sense.
If the initial data are well-prepared, i.e.,
\[  \divg \tilde{u}_0 =0, ~~\tilde{\rho}_0 + \tilde{\theta}_0 =0,~~  \]
then \eqref{well}
can be improved to 
\begin{align*}
\mathbf{P} u_\epsilon \to u,~~\tfrac{3}{5}\theta_\epsilon - \tfrac{2}{5}\rho_\epsilon \to \theta,~~\text{in},~~C([0, +\infty); H^{s-1}_x).
\end{align*}
\end{remark}
\begin{remark}
The main results of this work can be generalized to more collisional collisional kernels such as Fokker–Planck and Landau kernel, see \cite[Apendix A]{briant-2015-be-to-ns} for details.

\end{remark}
\subsection{The strategy  and difficulty of the proof.}
We only sketch the proof of Theorem \ref{theoremexi}. The proof of Theorem \ref{theoremlimit} is based on the local conservation laws of VMB system. For the uniform estimates, the goal is to obtain inequality like this
\begin{align*}
\tdt E(t) + D(t) \le E(t) D(t),
\end{align*}
where
\begin{align*}
E(t) \approx \|(f_\epsilon, g_\epsilon, E_\epsilon, B_\epsilon)\|_{H^s_x}^2 + \epsilon^2 \|(\nabla_v f_\epsilon, \nabla_v g_\epsilon)\|_{H^{s-1}}^2,\\~~D(t) \approx \|(f_\epsilon, g_\epsilon, \reps f_\epsilon^\perp, \reps g_\epsilon^\perp)\|_{H^s_\Lambda}^2 +  \tfrac{\alpha_\epsilon^2}{\epsilon^2}\|(B_\epsilon, E_\epsilon)\|_{H^{s-1}_x}^2.
\end{align*}
In what follows, we take the first equation of \eqref{vmbtwolinear}
\begin{align*}
\partial_t f_\epsilon + \reps \vdot f_\epsilon + \repst \mathcal{L}(f_\epsilon) = \cdots
\end{align*}
to explain the ``mixture norm'' skills used in \cite{mouhotneumann-2006-decay} and \cite{briant-2015-be-to-ns}. Because of the local coercivity properties of $\mathcal{L}$ and $\mathsf{L}$, we can obtain 
\begin{align*}
z_1 \sdot \dt \|f_\epsilon\|_{H^1_x}^2 + z_1 \sdot \repst \|f_\epsilon^\perp\|_{H^1_{\Lambda_x}}^2 \le \cdots.
\end{align*}
For the mixture term $\intps \nabla_x f_\epsilon \sdot \nabla_v f_\epsilon \bdv \bd x$,
\begin{align}
\label{est-z2}
z_2 \sdot \epsilon \dt \intps \nabla_x f_\epsilon \sdot \nabla_v f_\epsilon \bdv \bd x + z_2 \sdot \|\nabla_x f_\epsilon\|_{L^2}^2 \le \cdots.
\end{align}
In the above inequality, there exists the dissipative estimate of $f_\epsilon$. The $L^2$ estimate of $f_\epsilon$ can be obtained by Poincare's inequality and the local conservation laws. Furthermore, we also need to estimate $\nabla_v f_\epsilon$
\begin{align*}
z_3 \sdot \epsilon^2 \dt \|\nabla_v f_\epsilon\|_{L^2}^2 + z_3 \sdot  \|\nabla_v f_\epsilon\|_{L^2_\Lambda}^2 \le \cdots.
\end{align*}
Choosing proper constants $z_1$, $z_2$ and $z_3$ to let the ``mixture norm" satisfy
\begin{align*}
z_1   \|f_\epsilon\|_{H^1_x}^2 + z_2 \sdot \epsilon \intps \nabla_x f_\epsilon \sdot \nabla_v f_\epsilon \bdv \bd x +  z_3 \sdot \epsilon^2   \|\nabla_v f_\epsilon\|_{L^2}^2 \approx \|f_\epsilon\|_{H^1_\epsilon}^2,
\end{align*}
one can obtain that
\begin{align*}
\tdt \|f_\epsilon\|_{H^1_\epsilon}^2 + \|f_\epsilon\|_{H^1_\epsilon}^2 \le \cdots.
\end{align*}
The Lemma \ref{lemmaonlyx}, Lemma \ref{lemmamacrox} and Lemma \ref{lemmav} consist in establishing the similar above inequalities for $f_\epsilon$ and $g_\epsilon$. From the previous analysis, the key point of \cite{mouhotneumann-2006-decay} and \cite{briant-2015-be-to-ns} is to obtain \eqref{est-z2} and then employ the 	Poincare's inequality to recover zero-order dissipative estimates solutions.   But there exist new difficulty for the dimensionless system \eqref{vmbtwolinear}. Indeed, for the  Boltzmann system, if the mean value of the initial datum is zero, then the solution preserves this properties. But for \eqref{vmbtwolinear}, from the global conservation laws (see Lemma \ref{lemmamean}),  the mean value of $f_\epsilon$ and $g_\epsilon$ may be not zero. The Poincare's inequality can not be directly used. The idea is to employ the global conservation law to overcome this difficulty. The second difficulty comes from the singular term with coefficient $\repsb$ in \eqref{vmbtwolinear}. Indeed, this term is hard to bound , while $\beta_\epsilon = O(1)$. This difficulty can be dealt with  by splitting $f_\epsilon$ and $g_\epsilon$ into macroscopic parts and microscopic parts(see Remark \ref{remarkr2}). 

The estimates of electromagnetic parts are new in this wok. For the eletromagnetic field, as mentioned before, multiplying the second equation by $\tilde{v}\m$ and integrating the resulting equation over $\mathbb{R}^3$, it follows that
\begin{align}
\label{est-tilde-v}
-  \partial_t \tilde{j}_\epsilon + \sigma\tfrac{\alpha_\epsilon}{\epsilon^2} E_\epsilon -\reps \divg  \intv \tilde{v}\otimes v g_\epsilon  \bdv   - \repst j_\epsilon = \cdots.
\end{align}
The key point is to multiply $\tilde{v}$ other than $v$.  One of the advantage is that  the very singular term $\repst j_\epsilon$ can be canceled by employ the equation of $E_\epsilon$ (see \eqref{estcurle-cancel}). Based on the above equation, one can obtain the dissipative estimates of $E_\epsilon$(see Lemma \ref{lemma-curl-e}). For the dissipative estimates of $B_\epsilon$, by the third and forth equations in \eqref{vmbtwolinear}, one can obtain that
\begin{align*}
-\gamma_\epsilon \tdt \intt E_\epsilon \sdot \curl B_\epsilon \bd x + \|\curl B_\epsilon\|_{L^2}^2 - \|\curl E_\epsilon\|_{L^2}^2 \le \cdots.
\end{align*}
Since we have obtain the dissipative estimates of $E_\epsilon$, one can finally obtain the dissipative of $B_\epsilon$ by the Helmholtz decomposition and Poincare's inequality.  

 With the uniform estimate at our disposal, the limiting equations can be deduced from the local conservation of laws of Vlasov-Maxwell-Boltzmann by employing the properties of the linear Boltzmann operator and the structure of Vlasov-Maxwell-Boltzmann system. 
\section{A prior estimates}
\label{secEstimates}
This section is devoted to proving the existence of solutions to \eqref{vmbtwolinear}, i.e., Theorem \ref{theoremexi}.  The key ingredient is the uniform prior estimate of solutions.  The proof is quite involved. We split the whole proof into four lemmas. 
\subsection{The analysis of global conservation laws}   The mean value of the macroscopic part of $f_\epsilon$ and $g_\epsilon$ is necessary to obtain the dissipative estimates of $g_\epsilon$ and $f_\epsilon$.
\begin{lemma}
\label{lemmamean}
The classic solution of system \eqref{vmbtwolinear} enjoy the following   conservation law:
\begin{align}
\label{estlemmamean}
\begin{split}
\tdt \intt \left(u_\epsilon + \gamma_\epsilon E_\epsilon \times B_\epsilon\right)(t) \bd x = 0,\\
\tdt \intt \left(\theta_\epsilon +   \epsilon \sdot \tfrac{|E_\epsilon|^2 + |B_\epsilon|^2}{3} \right)(t) \bd x  =0,\\
\tdt \intt \rho_\epsilon(t) \bd x =\tdt \intt n_\epsilon(t) \bd x =\tdt \intt B_\epsilon(t) \bd x =0.
\end{split}
\end{align}
\end{lemma}
\begin{proof}
Recalling that
\[ v \sdot (v\times B_\epsilon) =0, \]
then we can rewrite system \eqref{vmbtwolinear} as
\begin{align}
\label{vmbtwo-rewrite}
\begin{cases}
\partial_t f_\epsilon +  \reps \vdot f_\epsilon  -   \repst \mathcal{L}(f_\epsilon)  = -  \tfrac{\alpha_\epsilon E_\epsilon + \beta_\epsilon v\times B_\epsilon}{\m \epsilon} \sdot \nabla_v (\m g_\epsilon) + \reps\Gamma(f_\epsilon,f_\epsilon), \\
\partial_t g_\epsilon +  \reps \vdot g_\epsilon   - \tfrac{\alpha_\epsilon}{\epsilon^2}  E_\epsilon \sdot v -   \repst \mathsf{L}(g_\epsilon)  =  -  \tfrac{\alpha_\epsilon E_\epsilon + \beta_\epsilon v\times B_\epsilon}{\m \epsilon} \sdot \nabla_v (\m f_\epsilon) + \reps\Gamma(g_\epsilon,f_\epsilon), \\
\ges\partial_t E_\epsilon - \curl B_\epsilon = - \tfrac{\beta_\epsilon}{\epsilon} j_\epsilon, \\
\ges  \partial_t B_\epsilon + \curl E_\epsilon =0,\\
\divg B_\epsilon =0,~~\divg E_\epsilon = \tfrac{\alpha_\epsilon}{\epsilon} \intv g_\epsilon\bdv.
\end{cases}
\end{align}
Then from \eqref{vmbtwo-rewrite}, it follows that
\begin{align}
\label{estMeanrho-n}
\tdt \intt \rho_\epsilon \bd x =0,~~\tdt \intt n_\epsilon(t) \bd x =0.
\end{align}
The local conservation law of velocity is
\begin{aligno}
\label{estConservationlaw-u}
\tdt \intt u_\epsilon(t) \bd x & = -   \intps v \sdot \left( \tfrac{\alpha_\epsilon E_\epsilon + \beta_\epsilon v\times B_\epsilon}{\epsilon} \sdot \nabla_v (\m g_\epsilon) \right) \bd v \bd x \\
& = \intps  \left( \tfrac{\alpha_\epsilon E_\epsilon + \beta_\epsilon v\times B_\epsilon}{\epsilon} \sdot   (\m g_\epsilon) \right) \bd v \bd x \\
& = \intt \divg E_\epsilon \sdot  E_\epsilon \bd x + \repsb \intt j_\epsilon\times B_\epsilon \bd x.
\end{aligno}
By simple computation,
\begin{aligno}
\label{useful}
 \divg (E_\epsilon \otimes E_\epsilon) & = \divg E_\epsilon \sdot E_\epsilon + (E_\epsilon \sdot \nabla) E_\epsilon,\\
 (\nabla \times E_\epsilon)\times E_\epsilon & = (E_\epsilon \sdot \nabla) E_\epsilon - \tfrac{1}{2}\nabla |E_\epsilon|^2,
\end{aligno}
plugging \eqref{useful} into \eqref{estConservationlaw-u}, we can infer that
\begin{aligno}
\label{estConservationlaw-u-1}
\tdt \intt u_\epsilon(t) \bd x = \repsb \intt j_\epsilon\times B_\epsilon \bd x - \intt (E_\epsilon \sdot \nabla) E_\epsilon \bd x
\end{aligno}
Multiplying the third equation of \eqref{vmbtwo-rewrite} by $\times B_\epsilon$ and the forth equation by $\times E_\epsilon$ respectively, then it follows that
\begin{aligno}
\label{estConservationlaw-u-2}
\gamma_\epsilon \tdt \intt E_\epsilon\times B_\epsilon \bd x  - \intt (\nabla \times B_\epsilon) \times B_\epsilon  + (\nabla \times E_\epsilon) \times E_\epsilon \bd x = - \repsb \intt j_\epsilon \times B_\epsilon \bd x.
\end{aligno}
With the help of  \eqref{useful} and $\divg B_\epsilon =0$, we can infer that
\begin{aligno}
\label{estecurlb}
  (\nabla \times B_\epsilon) \times B_\epsilon  + (\nabla \times E_\epsilon) \times E_\epsilon  & = (E_\epsilon \sdot \nabla) E_\epsilon + (B_\epsilon \sdot \nabla) B_\epsilon - \tfrac{1}{2}\nabla \left(  |E_\epsilon|^2 + |B_\epsilon|^2\right)\\
  & = (E_\epsilon \sdot \nabla) E_\epsilon + (B_\epsilon \sdot \nabla) B_\epsilon - \tfrac{1}{2}\nabla \left(  |E_\epsilon|^2 + |B_\epsilon|^2\right) \\
  & = (E_\epsilon \sdot \nabla) E_\epsilon - \divg (B_\epsilon \otimes B_\epsilon)  - \tfrac{1}{2}\nabla \left(  |E_\epsilon|^2 + |B_\epsilon|^2\right).
\end{aligno}
Combining \eqref{estecurlb}, \eqref{estConservationlaw-u-1} and \eqref{estConservationlaw-u-2}, we can infer that
\begin{align}
\label{estMeanu}
\tdt \intt \left(u_\epsilon + \gamma_\epsilon E_\epsilon \times B_\epsilon\right)(t) \bd x = 0.
\end{align}
Multiplying the first equation of \eqref{vmbtwo-rewrite} by $\tfrac{|v|^2-3}{3}$ and then integrating over the phase space, we can infer that
\begin{aligno}
\label{estMeantheta-1}
\tdt \intt \theta_\epsilon(t) \bd x & = -   \intps \tfrac{|v|^2-3}{3}  \sdot \left( \tfrac{\alpha_\epsilon E_\epsilon + \beta_\epsilon v\times B_\epsilon}{\epsilon} \sdot \nabla_v (\m g_\epsilon) \right) \bd v \bd x \\
& = \tfrac{2}{3} \intps  \left( \tfrac{v \cdot (\alpha_\epsilon E_\epsilon + \beta_\epsilon v\times B_\epsilon)}{\epsilon}     (\m g_\epsilon) \right) \bd v \bd x \\
& = \tfrac{2 \beta_\epsilon }{3  \gamma_\epsilon} \intt E_\epsilon \sdot j_\epsilon \bd x.
\end{aligno}
By the similar trick of deducing \eqref{estConservationlaw-u-2}, we can infer that
\begin{aligno}
\label{estMeantheta-2}
\epsilon \dt \intt |E_\epsilon|^2 + |B_\epsilon|^2 \bd x   =  \tfrac{  \beta_\epsilon }{  \gamma_\epsilon} \intt E_\epsilon \sdot j_\epsilon \bd x.
\end{aligno}
Combining \eqref{estMeantheta-1} and \eqref{estMeantheta-2}, we can infer that
\begin{aligno}
\label{estMeanstheta}
\dt \intt \left(\theta_\epsilon +   \epsilon \sdot \tfrac{|E_\epsilon|^2 + |B_\epsilon|^2}{3} \right)(t) \bd x  =0.
\end{aligno}
\end{proof}

Since the proof of uniform estimates  is very involved, we split it into serval parts.

\subsection{The dissipative estimates of the microscopic part}

\begin{lemma}[only related to $\nabla_x^k$]
\label{lemmaonlyx}
Under the assumption  in the section \ref{sec-assump-on-L} and the assumption \eqref{estMeaninitial} on the initial data, if $(f_\epsilon, g_\epsilon, B_\epsilon, E_\epsilon)$ are solutions to \eqref{vmbtwolinear}, then 
\begin{aligno}
\label{estlemmaonlyx}
& \dt \|  (f_\epsilon, g_\epsilon, E_\epsilon, B_\epsilon)\|_{H^s_x}^2  + \tfrac{3}{4\epsilon^2}   \| ( f^\perp_\epsilon,  g^\perp_\epsilon)\|_{H^s_{\Lambda_x}}^2    \\
& \le   C \|(E_\epsilon, B_\epsilon, f_\epsilon, g_\epsilon)\|_{H^s_x}^2\|(f_\epsilon, g_\epsilon)\|_{H^s_\Lambda}^2.
\end{aligno}

\end{lemma}
\begin{remark}
\label{remarkr2}
The term $ \repsb v \times B_\epsilon \sdot \nabla_v g_\epsilon$ brings new difficulties. Indeed, 
There already exists derivative with respect to $v$ in the Lorentz force. While $\beta_\epsilon = O(1)$, $\repsb = \reps$. The idea to deal with this difficulty is to decompose $f_\epsilon$ and $g_\epsilon$ into fluid part and microscopic part (see \eqref{estr222}).
\end{remark}
\begin{proof}
Applying $\nabla_x^k $ to the first four equations of \eqref{vmbtwo-rewrite} and then multiplying the resulting equations by $\nabla^k_x f_\epsilon \m $, $\nabla_x^k g_\epsilon \m $, $\nabla_x^k E_\epsilon \m$ and $\nabla_x^k B_\epsilon \m $, the integration over the phase space leads to
\begin{aligno}
\label{estx-0}
& \dt \|\nabla^k_x (f_\epsilon, g_\epsilon, E_\epsilon, B_\epsilon)\|_{L^2}^2  - \tfrac{\alpha_\epsilon}{\epsilon^2} \intt E_\epsilon \sdot j_\epsilon \bd x  +  \tfrac{\beta_\epsilon}{\epsilon \gamma_\epsilon} \intt E_\epsilon \sdot j_\epsilon \bd x \\
& - \repst  \intps  \left( \mathcal{L}(\nabla_x^k f_\epsilon)\sdot \nabla_x^k f_\epsilon    +  \mathsf{L}(\nabla_x^k g_\epsilon)\sdot \nabla_x^k g_\epsilon \right) \bdv \bd x\\
& =   -\intps \left( \nabla_x^k \left(    \tfrac{\alpha_\epsilon E_\epsilon + \beta_\epsilon v\times B_\epsilon}{\m \epsilon} \sdot \nabla_v ( \m g_\epsilon)\right) \nabla_x^k f_\epsilon  +  \nabla_x^k \left(\tfrac{\alpha_\epsilon E_\epsilon + \beta_\epsilon v\times B_\epsilon}{\m \epsilon} \sdot \nabla_v ( \m f_\epsilon) \right)\nabla_x^k g_\epsilon \right) \bdv \bd x \\
& + \reps \intps \left( \nabla_x^k   \Gamma(f_\epsilon,f_\epsilon)  \sdot \nabla_x^k f_\epsilon  +  \nabla_x^k \Gamma(g_\epsilon,f_\epsilon)  \sdot \nabla_x^k g_\epsilon \right)   \bdv \bd x \\
& := D_1 + D_n.
\end{aligno}
Recalling that $\alpha_\epsilon \gamma_\epsilon =  \epsilon \beta_\epsilon$ and by the assumptions on $\mathcal{L}$ and $\mathsf{L}$, we can infer that
\begin{aligno}
\label{estx-1}
& \dt \|\nabla^k_x (f_\epsilon, g_\epsilon, E_\epsilon, B_\epsilon)\|_{L^2}^2  + \repst \left( \|\nabla_x^k f^\perp\|_{L^2_\Lambda}^2 + \|\nabla_x^k g^\perp\|_{L^2_\Lambda}^2  \right) \le   D_1 + D_n.
\end{aligno}
Now, we try to estimate $D_1$ first. For $\beta=1$,   $\repsb = \epsilon^{-1}$ is unbounded while $\epsilon \to 0$. Since $\alpha \le \epsilon$, $\repsa \le 1$. Thus we only need to pay attention to the term with coefficient $\repsb$. Secondly, while $k=s$ and all the derivative acts on $g_\epsilon$ and $f_\epsilon$, that is to say,
\[ \intps \left(  \left(    \tfrac{\alpha_\epsilon E_\epsilon + \beta_\epsilon v\times B_\epsilon}{\m \epsilon} \sdot  ( \m \nabla_x^k \nabla_v g_\epsilon)\right) \nabla_x^k f_\epsilon  +    \left(\tfrac{\alpha_\epsilon E_\epsilon + \beta_\epsilon v\times B_\epsilon}{\m \epsilon} \sdot  ( \m \nabla_x^k \nabla_v f_\epsilon) \right)\nabla_x^k g_\epsilon \right) \bdv \bd x, \]
the above term can not be directly controlled. To overcome these difficulties, we first split
\begin{aligno}
\label{estonlyD1-0}
D_1 &  = \intps \left( \left(    \tfrac{\alpha_\epsilon E_\epsilon + \beta_\epsilon v\times B_\epsilon}{\m \epsilon} \sdot \nabla_v  \nabla_x^k ( \m g_\epsilon)\right) \nabla_x^k f_\epsilon    \right) \bdv \bd x \\
 & +  \intps \left(      \nabla_x^k \left(\tfrac{\alpha_\epsilon E_\epsilon + \beta_\epsilon v\times B_\epsilon}{\m \epsilon} \sdot \nabla_v  \nabla_x^k ( \m f_\epsilon) \right)\nabla_x^k g_\epsilon \right) \bdv \bd x + D_r \\
& := D_2 + D_r.
 \end{aligno}
$D_2$ is simple and can be bounded by integrating by parts over the phase space. Indeed, by simple computation,  we can conclude that
\begin{aligno}
\label{estonlyxD2}
D_2 &  = \intps   \left(    \tfrac{\alpha_\epsilon E_\epsilon + \beta_\epsilon v\times B_\epsilon}{  \epsilon} \sdot \nabla_x^k  \nabla_v ( \m (f_\epsilon + g_\epsilon)\right) \nabla_x^k (f_\epsilon  + g_\epsilon)  \bd v \bd x \\
 & - \intps    \left(    \tfrac{\alpha_\epsilon E_\epsilon + \beta_\epsilon v\times B_\epsilon}{  \epsilon} \sdot  \nabla_x^k \nabla_v ( \m  f_\epsilon  )\right) \nabla_x^k (f_\epsilon  )  \bd v \bd x \\
 & - \intps    \left(    \tfrac{\alpha_\epsilon E_\epsilon + \beta_\epsilon v\times B_\epsilon}{  \epsilon} \sdot \nabla_x^k \nabla_v ( \m  g_\epsilon  )\right) \nabla_x^k (g_\epsilon  )  \bd v \bd x\\
 & =  - \tfrac{1}{2}\intps   \left(    \tfrac{\alpha_\epsilon E_\epsilon + \beta_\epsilon v\times B_\epsilon}{  \epsilon} \sdot    v \sdot    \nabla_x^k (f_\epsilon + g_\epsilon)\right) \nabla_x^k (f_\epsilon  + g_\epsilon)  \bdv \bd x \\
 & + \tfrac{1}{2}\intps    \left(    \tfrac{\alpha_\epsilon E_\epsilon + \beta_\epsilon v\times B_\epsilon}{ \epsilon} \sdot    v \sdot   \nabla_x^k f_\epsilon  )\right) \nabla_x^k (f_\epsilon  )  \bdv \bd x \\
 & + \tfrac{1}{2} \intps    \left(    \tfrac{\alpha_\epsilon E_\epsilon + \beta_\epsilon v\times B_\epsilon}{\epsilon} \sdot   v (  \nabla_x^k g_\epsilon  )\right) \nabla_x^k (g_\epsilon  )  \bdv \bd x\\
 & = -   \tfrac{ \alpha_\epsilon}{  \epsilon}\intps         v \sdot  E_\epsilon \nabla_x^k f_\epsilon     \nabla_x^k (g_\epsilon  )  \bd v \bd x\\
 & \lesssim \repsa \|E_\epsilon\|_{H^s}\|f_\epsilon\|_{H^s_{\Lambda_x}} \|g_\epsilon\|_{H^s_{\Lambda_x}}
\end{aligno}
To deal with $D_r$,  \begin{aligno}
D_r &  = \sum\limits_{ i \ge 1 \atop i+j=k} \intps   \nabla_x^i\left(    \tfrac{\alpha_\epsilon E_\epsilon + \beta_\epsilon v\times B_\epsilon}{\m \epsilon}\right) \sdot \nabla_v  \nabla_x^j ( \m g_\epsilon) \nabla_x^k f_\epsilon      \bdv \bd x \\
 & + \sum\limits_{ i \ge 1 \atop i+j=k} \intps        \nabla_x^i \left(\tfrac{\alpha_\epsilon E_\epsilon + \beta_\epsilon v\times B_\epsilon}{\m \epsilon}\right) \sdot \nabla_v  \nabla_x^j ( \m f_\epsilon) \nabla_x^k g_\epsilon   \bdv \bd x\\
 & := D_{r1} + D_{r2}.
 \end{aligno}
We need to use the structure of $g_\epsilon$ to control $D_{r2}$. Indeed, noticing that
\begin{align*}
g_\epsilon(t,x,v) = n_\epsilon(t,x) + g_\epsilon^\perp(t,x,v),
\end{align*}
thus, we can infer that $\nabla_v \bdp g_\epsilon  =0$. With the help of this fact, we can infer that
\begin{aligno}
\label{estr222}
D_{r2} &  = \sum\limits_{ i \ge 1 \atop i+j=k} \intps        \nabla_x^i \left(\tfrac{\alpha_\epsilon E_\epsilon + \beta_\epsilon v\times B_\epsilon}{  \epsilon}\right) \sdot \nabla_v  \nabla_x^j ( \m f_\epsilon) \nabla_x^k g_\epsilon   \bd v \bd x \\
& = \sum\limits_{ i \ge 1 \atop i+j=k} \intps        \nabla_x^i \left(\tfrac{\alpha_\epsilon E_\epsilon + \beta_\epsilon v\times B_\epsilon}{ \epsilon}\right) \sdot \nabla_v  \nabla_x^j ( \m  f_\epsilon) \nabla_x^k g_\epsilon^\perp  \bd v \bd x \\
& + \sum\limits_{ i \ge 1 \atop i+j=k} \intps        \nabla_x^i \left(\tfrac{\alpha_\epsilon E_\epsilon + \beta_\epsilon v\times B_\epsilon}{  \epsilon}\right) \sdot \nabla_v  \nabla_x^j ( \m  f_\epsilon) \nabla_x^k \bdp g_\epsilon   \bd v \bd x \\
& = \sum\limits_{ i \ge 1 \atop i+j=k} \intps        \nabla_x^i \left(\tfrac{\alpha_\epsilon E_\epsilon + \beta_\epsilon v\times B_\epsilon}{ \epsilon}\right) \sdot \nabla_v  \nabla_x^j ( \m  f_\epsilon) \nabla_x^k g_\epsilon^\perp  \bd v \bd x\\
& = - \repsa\sum\limits_{ i \ge 1 \atop i+j=k} \intps     v\sdot   \nabla_x^i E_\epsilon      \nabla_x^j  f_\epsilon  \nabla_x^k g_\epsilon^\perp  \bdv \bd x \\
& + \repsa\sum\limits_{ i \ge 1 \atop i+j=k} \intps        \nabla_x^i E_\epsilon     \sdot \nabla_v \nabla_x^j  f_\epsilon \cdot \nabla_x^k g_\epsilon^\perp  \bdv \bd x \\
& + \repsb\sum\limits_{ i \ge 1 \atop i+j=k} \intps       v\times \nabla_x^i B_\epsilon     \sdot \nabla_v \nabla_x^j  f_\epsilon \cdot \nabla_x^k g_\epsilon^\perp  \bdv \bd x\\
& =D_{r21} + D_{r22} + D_{r23}.
\end{aligno}
The way of controling $D_{r21}$, $D_{r22}$ and $D_{r23}$ are similar. Here, we only take $D_{r23}$ for example.
\begin{align*}
D_{r23} & \le \repsb \sum\limits_{ i \ge 1 \atop i+j=k} \intt |\nabla_x^i B_\epsilon| \intv          |v|   |\nabla_v \nabla_x^j  f_\epsilon| \cdot |\nabla_x^k g_\epsilon^\perp|  \bdv \bd x \\
& \le  \repsb \sum\limits_{ i \ge 1 \atop i+j=k} \intt |\nabla_x^i B_\epsilon(t,x)|   \|\nabla_v \nabla_x^j  f_\epsilon(t,x)\|_{L^2_{\Lambda_v}}   \|  \nabla_x^k g_\epsilon^\perp(t,x)\|_{L^2_{\Lambda_v}} \bd x \\
&  \le \repsb \sum\limits_{ [\tfrac{s}{2}] \ge i \ge 1 \atop i+j=k} \intt |\nabla_x^i B_\epsilon(t,x)|   \|\nabla_v \nabla_x^j  f_\epsilon(t,x)\|_{L^2_{\Lambda_v}}   \|  \nabla_x^k g_\epsilon^\perp(t,x)\|_{L^2_{\Lambda_v}} \bd x \\
& + \repsb \sum\limits_{ [\tfrac{s}{2}] \le i \le s \atop i+j=k} \intt |\nabla_x^i B_\epsilon(t,x)|   \|\nabla_v \nabla_x^j  f_\epsilon(t,x)\|_{L^2_{\Lambda_v}}   \|  \nabla_x^k g_\epsilon^\perp(t,x)\|_{L^2_{\Lambda_v}} \bd x \\
& \le \repsb \| \|\nabla_v \nabla_x^j  f_\epsilon(t,x)\|_{L^2_{\Lambda_v}}\|_{L^\infty_x}  \sum\limits_{ [\tfrac{s}{2}] \le i \le s \atop i+j=k} \intt |\nabla_x^i B_\epsilon(t,x)|     \|  \nabla_x^k g_\epsilon^\perp(t,x)\|_{L^2_{\Lambda_v}} \bd x \\
& + \repsb \|\nabla_x^i B_\epsilon(t,x)\|_{L^\infty_x} \sum\limits_{ [\tfrac{s}{2}] \ge i \ge 1 \atop i+j=k} \intt    \|\nabla_v \nabla_x^j  f_\epsilon(t,x)\|_{L^2_{\Lambda_v}}   \|  \nabla_x^k g_\epsilon^\perp(t,x)\|_{L^2_{\Lambda_v}} \bd x.
\end{align*}
Then by the Sobolev embedding inequalities (for $x$) and H\"older inequality, we can infer
\begin{aligno}
\label{estonlyxD2r}
D_{r2} \le C \|(E_\epsilon, B_\epsilon)\|_{H^s_x}^2 \|f_\epsilon\|_{H^s_\Lambda}^2 + \tfrac{1}{8\epsilon^2} \|\nabla^k_x g_\epsilon^\perp\|_{L^2_{\Lambda}}^2.
\end{aligno}
The trick of controlling $D_{r1}$ is very complicated.  Except for employing the structure of $g_\epsilon$, we also need to split $f_\epsilon$ into macroscopic part and microscopic part. Indeed,
\begin{aligno}
\label{estonlyxD1r-1}
D_{r1} & =  \sum\limits_{ i \ge 1 \atop i+j=k} \intps   \nabla_x^i\left(    \tfrac{\alpha_\epsilon E_\epsilon + \beta_\epsilon v\times B_\epsilon}{  \epsilon}\right) \sdot \nabla_v  \nabla_x^j ( \m g_\epsilon) \nabla_x^k f_\epsilon^\perp     \bd v \bd x \\
& + \sum\limits_{ i \ge 1 \atop i+j=k} \intps   \nabla_x^i\left(    \tfrac{\alpha_\epsilon E_\epsilon + \beta_\epsilon v\times B_\epsilon}{  \epsilon}\right) \sdot \nabla_v  \nabla_x^j ( \m n_\epsilon) \nabla_x^k \bdp f_\epsilon     \bd v \bd x \\
& + \sum\limits_{ i \ge 1 \atop i+j=k} \intps   \nabla_x^i\left(    \tfrac{\alpha_\epsilon E_\epsilon + \beta_\epsilon v\times B_\epsilon}{  \epsilon}\right) \sdot \nabla_v  \nabla_x^j ( \m g_\epsilon^\perp) \nabla_x^k \bdp f_\epsilon     \bd v \bd x.
\end{aligno}
Except for the second term in the right hand of \eqref{estonlyxD1r-1}, the other two terms are easy to be bounded. Indeed, for the first term in \eqref{estonlyxD1r-1}, we can infer that
\begin{aligno}
\label{estonlyxD1r-2}
& \sum\limits_{ i \ge 1 \atop i+j=k} \intps   \nabla_x^i\left(    \tfrac{\alpha_\epsilon E_\epsilon + \beta_\epsilon v\times B_\epsilon}{  \epsilon}\right) \sdot \nabla_v  \nabla_x^j ( \m g_\epsilon) \nabla_x^k f_\epsilon^\perp     \bd v \bd x \\
& = -\repsa \sum\limits_{ i \ge 1 \atop i+j=k} \intps        v  \sdot \nabla_x^i E_\epsilon  \sdot  \nabla_x^j   g_\epsilon  \nabla_x^k f_\epsilon^\perp     \bdv \bd x \\
& + \repsa \sum\limits_{ i \ge 1 \atop i+j=k} \intps          \nabla_x^i E_\epsilon  \sdot \nabla_v \nabla_x^j   g_\epsilon  \nabla_x^k f_\epsilon^\perp     \bdv \bd x \\
& +  \repsb \sum\limits_{ i \ge 1 \atop i+j=k} \intps         v\times \nabla_x^i B_\epsilon  \sdot \nabla_v \nabla_x^j   g_\epsilon  \nabla_x^k f_\epsilon^\perp     \bdv \bd x \\
& \le  C \|(E_\epsilon, B_\epsilon)\|_{H^s_x}^2\|(f_\epsilon, g_\epsilon)\|_{H^s_\Lambda}^2 + \tfrac{1}{16\epsilon^2}\|\nabla_x^k f_\epsilon^\perp\|_{L^2_\Lambda}^2.
\end{aligno}
Similarly, we can infer that
\begin{aligno}
\label{estonlyxD1r-3}
\sum\limits_{ i \ge 1 \atop i+j=k} \intps   \nabla_x^i\left(    \tfrac{\alpha_\epsilon E_\epsilon + \beta_\epsilon v\times B_\epsilon}{  \epsilon}\right) \sdot \nabla_v  \nabla_x^j ( \m g_\epsilon^\perp) \nabla_x^k \bdp f_\epsilon     \bd v \bd x \\
 \le C \|(E_\epsilon, B_\epsilon)\|_{H^s_x}^2\|(f_\epsilon, g_\epsilon)\|_{H^s_\Lambda}^2 + \tfrac{1}{16\epsilon^2}\|\nabla_x^k f_\epsilon^\perp\|_{L^2_\Lambda}^2.
\end{aligno}
The second term in \eqref{estonlyxD1r-1} is more direct. Recalling that
\[\bdp f_\epsilon = \rho_\epsilon + u_\epsilon \sdot v +  \theta_\epsilon \tfrac{|v|^2-3}{2},~~~~ \nabla_v n_\epsilon =0,    \]
then it follows that
\begin{aligno}
\label{estonlyxD1r-4}
& \sum\limits_{ i \ge 1 \atop i+j=k} \intps   \nabla_x^i\left(    \tfrac{\alpha_\epsilon E_\epsilon + \beta_\epsilon v\times B_\epsilon}{  \epsilon}\right) \sdot \nabla_v  \nabla_x^j ( \m n_\epsilon) \nabla_x^k \bdp f_\epsilon     \bd v \bd x \\
& = - \repsa \sum\limits_{ i \ge 1 \atop i+j=k} \intps \left( v \sdot  \nabla_x^i E_\epsilon \right) \sdot   \nabla_x^j  n_\epsilon \nabla_x^k \bdp f_\epsilon     \bdv \bd x \\
& \le C \|(E_\epsilon, B_\epsilon, f_\epsilon, g_\epsilon)\|_{H^s_x}^2\|(f_\epsilon, g_\epsilon)\|_{H^s_x}^2.
\end{aligno}
With the help of \eqref{estonlyxD1r-1}, \eqref{estonlyxD1r-2}, \eqref{estonlyxD1r-3}, \eqref{estonlyxD1r-4} and \eqref{estonlyxD2r}, we can finally conclude
\begin{aligno}
\label{estonlyxD1r}
D_{1r} \le C \|(E_\epsilon, B_\epsilon, f_\epsilon, g_\epsilon)\|_{H^s_x}^2\|(f_\epsilon, g_\epsilon)\|_{H^s_\Lambda}^2 + \tfrac{1}{8\epsilon^2}\|\nabla_x^k f_\epsilon^\perp\|_{L^2_\Lambda}^2,
\end{aligno}
and
\begin{aligno}
\label{estonlyxDr}
D_{r} \le C \|(E_\epsilon, B_\epsilon, f_\epsilon, g_\epsilon)\|_{H^s_x}^2\|(f_\epsilon, g_\epsilon)\|_{H^s_\Lambda}^2 + \tfrac{1}{8\epsilon^2}\|\nabla_x^k (f_\epsilon^\perp, g_\epsilon^\perp)\|_{L^2_\Lambda}^2.
\end{aligno}
Finally, with the help of \eqref{estonlyxD2} and \eqref{estonlyxDr},  for the $D_1$ in \eqref{estx-1}, we can obtain that
\begin{align}
\label{estonlyxD1}
D_1 \le C \|(E_\epsilon, B_\epsilon, f_\epsilon, g_\epsilon)\|_{H^s_x}^2\|(f_\epsilon, g_\epsilon)\|_{H^s_\Lambda}^2 + \tfrac{1}{8\epsilon^2}\|\nabla_x^k (f_\epsilon^\perp, g_\epsilon^\perp)\|_{L^2_\Lambda}^2.
\end{align}
For the nonlinear collision operator ($D_n$ in \eqref{estx-1}),  by the assumption, we can infer that
\begin{aligno}
\label{estonlyxDn}
D_n & =  \reps \intps \left( \nabla_x^k   \Gamma(f_\epsilon,f_\epsilon)  \sdot \nabla_x^k f_\epsilon^\perp  +  \nabla_x^k \Gamma(g_\epsilon,f_\epsilon)  \sdot \nabla_x^k g_\epsilon^\perp \right)   \bdv \bd x \\
& \le C\|(f_\epsilon, g_\epsilon)\|_{H_x^s}\|(f_\epsilon, g_\epsilon)\|_{H_{\Lambda_x}^s}\|\nabla^k_x (f_\epsilon, g_\epsilon)\|_{L_{\Lambda}^2} \\
& \le C \|(f_\epsilon, g_\epsilon)\|_{H_x^s}^2\|(f_\epsilon, g_\epsilon)\|_{H_{\Lambda_x}^s}^2 +  \tfrac{1}{8\epsilon^2} \|\nabla^k_x (f_\epsilon, g_\epsilon)\|_{L_{\Lambda}^2}.
\end{aligno}
Combing \eqref{estx-0}, \eqref{estx-1}, \eqref{estonlyxD1} and \eqref{estonlyxDn}, we complete the proof of this lemma.

\end{proof}

  Lemma \ref{lemmaonlyx} is only related to the derivative of $f_\epsilon$ and $g_\epsilon$ with respect to $x$. The following lemma is devoted to the estimates of $\nabla_x^i \nabla_x^j g_\epsilon$ ($i \ge 1$).
Loosely speaking,   the goal of the following lemma is to obtain estimates like this
\begin{align*}
 \epsilon^2 \dt \| (\nabla_v f_\epsilon, \nabla_v g_\epsilon)\|_{H^{s-1}}^2 +\| (\nabla_v f_\epsilon, \nabla_v g_\epsilon)\|_{H^{s-1}_\Lambda}^2 \le \|(f_\epsilon,g_\epsilon, \alpha_\epsilon E_\epsilon)\|_{H^{s-1}_x}^2 +  \cdots,~~ s = i+ j, ~~i \ge 1.
\end{align*}
  Since there exists a coefficient $\epsilon^2$ before $\dt \|(\nabla_x^j \nabla^i_v f_\epsilon, \nabla_x^j \nabla^i_v g_\epsilon)\|_{L^2}^2$. Compared to Lemma \ref{lemmaonlyx},   it is simpler. Indeed, the singular term (the one with coefficient $\repsb$) is easily estimated.
   Before stating this lemma, we introduce the equivalent norm of $\dot{H}^k$
 \[ \dot{H}_{v,\epsilon}^{k}(t) =    8 \sum\limits_{ i \ge 1, j \ge 1 \atop i + j =k}\|( \nabla_v^i \nabla_x^j f_\epsilon,  \nabla_v^i \nabla_x^j g_\epsilon)\|_{L^2}^2  + \|( \nabla_v^k   f_\epsilon,  \nabla_v^k g_\epsilon)\|_{L^2}^2. \]
\begin{lemma}
\label{lemmav}
Under the assumption  in the section \ref{sec-assump-on-L} and the assumption \eqref{estMeaninitial} on the initial data, if $(f_\epsilon, g_\epsilon, B_\epsilon, E_\epsilon)$ are solutions to \eqref{vmbtwolinear}, then 
\begin{aligno}
\label{estv}
& \epsilon^2 \dt \left( \tfrac{8c_1}{3} \sum\limits_{m=1}^{s-1} \dot{H}_{v,\epsilon}^{m}(t)  + \dot{H}_{v,\epsilon}^{s}(t)\right)  +  \tfrac{3}{4} \|(\nabla_v f_\epsilon, \nabla_v g_\epsilon)\|_{H^{s-1}_\Lambda}^2\\
& \le b_1 \|(f_\epsilon, g_\epsilon)\|_{H^{s-1}_x}^2  +  b_1 \alpha_\epsilon^2\| E_\epsilon\|_{H^{s-1}_x}^2  + C \|(f_\epsilon, g_\epsilon, E_\epsilon, B_\epsilon)\|_{H^s_x}\|(f_\epsilon, g_\epsilon)\|_{H^s_\Lambda}^2 \\
& +  C \sdot \epsilon   \|(\nabla_v f_\epsilon, \nabla_v g_\epsilon)\|_{H^{s-1}}\|(f_\epsilon, g_\epsilon)\|_{H^s_\Lambda}^2,
\end{aligno}
where $c_1$ comes from the computation and is only dependent of the Sobolev embedding constants.
\end{lemma}
\begin{proof}
Applying $\nabla_x^j\nabla_v^i$ to the first two equations of \eqref{vmbtwolinear}, then multiplying the resulting equations by $ \epsilon^2 \nabla_x^j \nabla_v^i f_\epsilon \m $ and  $ \epsilon^2 \nabla_x^j \nabla_v^i g_\epsilon \m $ respectively, the integration over $\phs$ leads to
\begin{aligno}
\label{estv-0}
& \epsilon^2 \dt \|(\nabla_x^j\nabla_v^i f_\epsilon, \nabla_x^j\nabla_v^i g_\epsilon)\|_{L^2}^2 + \intps \left( \nabla_x^j \nabla_v^i  \mathcal{L}(f_\epsilon) \sdot \nabla_x^j\nabla_v^i f_\epsilon  +  \nabla_x^j \nabla_v^i  \mathsf{L}(g_\epsilon) \sdot \nabla_x^j\nabla_v^i g_\epsilon \right) \bdv \bd x\\
& = - \epsilon \intps \left(  \nabla_x^j \nabla_v^i ( \vdot f_\epsilon) \sdot \nabla_x^j\nabla_v^i f_\epsilon  + \nabla_x^j \nabla_v^i ( \vdot g_\epsilon) \sdot \nabla_x^j\nabla_v^i g_\epsilon \right) \bdv \bd x \\
& \quad
 + \epsilon \intps \left(\nabla_x^j\nabla_v^i N_f \sdot \nabla_x^j \nabla_v^i f_\epsilon   + \nabla_x^j\nabla_v^i N_g \sdot \nabla_x^j \nabla_v^i g_\epsilon \right) \bdv \bd x \\
& \quad  + \alpha_\epsilon \intps \nabla_x^j \nabla_v^i (v \sdot E_\epsilon) \sdot \nabla_x^j \nabla_v^i g_\epsilon \bdv \bd x\\
& = T_1 + T_2 + T_3.
\end{aligno}
For the left hand of \eqref{estv-0},  by the assumption on the linear Boltzmann operator, we can infer that
\begin{aligno}
\label{estvleft}
& \tfrac{15}{16}\|(\nabla_x^j \nabla_v^i f_\epsilon, \nabla_x^j \nabla_v^i g_\epsilon)\|_{L^2_\Lambda}^2  -  C \|(   f_\epsilon,   g_\epsilon)\|_{H^{k-1}}^2  \\
 &  \le \intps \left( \nabla_x^j \nabla_v^i  \mathcal{L}(f_\epsilon) \sdot \nabla_x^j\nabla_v^i f_\epsilon  +  \nabla_x^j \nabla_v^i  \mathsf{L}(g_\epsilon) \sdot \nabla_x^j\nabla_v^i g_\epsilon \right) \bdv \bd x.
\end{aligno}
We first deal with the quadratic term ($T_1$ and $T_3$) in  the right hand of \eqref{estv-0}. For $T_3$,  by H\"older inequality, we can infer that
\begin{align}
\label{estvt3}
T_3 \le 4 \alpha_\epsilon^2\|\nabla_x^j \nabla_v^i (v  E_\epsilon)\|_{L^2}^2 + \tfrac{1}{16} \|\nabla_x^j \nabla_v^i g_\epsilon\|_{L^2}^2.
\end{align}
For $T_1$, by integration by parts over the phase space, we can infer that
\begin{aligno}
\label{estvt1}
T_1  & \le  \epsilon \|\nabla_v^{i-1} \nabla_x^{j+1} f_\epsilon\|_{L^2}\|\nabla_v^{i} \nabla_x^{j} f_\epsilon\|_{L^2} +  \epsilon \|\nabla_v^{i-1} \nabla_x^{j+1} g_\epsilon\|_{L^2}\|\nabla_v^{i} \nabla_x^{j} g_\epsilon\|_{L^2} \\
& \le 4 \|\nabla_v^{i-1} \nabla_x^{j+1} (f_\epsilon, g_\epsilon)\|_{L^2}^2 + \tfrac{1}{16} \|\nabla_v^{i} \nabla_x^{j} (f_\epsilon, g_\epsilon)\|_{L^2}^2.
\end{aligno}
For the triple term $T_2$, the way of controlling this term is similar to that of $D_1$ and $D_n$ in Lemma \ref{lemmaonlyx}. Recalling that
\begin{align*}
  N_f & =  \tfrac{\alpha_\epsilon}{\epsilon} E_\epsilon\sdot v \sdot g_\epsilon -  ( \repsa E_\epsilon + \repsb v \times B_\epsilon) \sdot \nabla_v g_\epsilon + \reps \Gamma(f_\epsilon, f_\epsilon), \\
  N_g & =  \tfrac{\alpha_\epsilon}{\epsilon} E_\epsilon\sdot v \sdot f_\epsilon - (\repsa E_\epsilon +  \repsb v \times B_\epsilon) \sdot \nabla_v f_\epsilon + \reps \Gamma(g_\epsilon, f_\epsilon),
\end{align*}
we can split $T_2$ as follows
\begin{aligno}
\label{estvt2-0}
T_2 & = \epsilon \intps \nabla_v^j \nabla_v^i \left(   {\alpha_\epsilon}  E_\epsilon\sdot v \sdot g_\epsilon -  ( \alpha_\epsilon E_\epsilon + \beta_\epsilon v \times B_\epsilon) \sdot \nabla_v g_\epsilon \right) \sdot \nabla_v^j \nabla_x^i f_\epsilon \bdv \bd x \\
& + \epsilon \intps \nabla_v^j \nabla_v^i \left(   {\alpha_\epsilon}  E_\epsilon\sdot v \sdot f_\epsilon -  ( \alpha_\epsilon E_\epsilon + \beta_\epsilon v \times B_\epsilon) \sdot \nabla_v f_\epsilon \right) \sdot \nabla_v^j \nabla_x^i g_\epsilon \bdv \bd x \\
& + \epsilon \intps \left( \nabla_x^j \nabla_v^i   \Gamma(f_\epsilon,f_\epsilon)  \sdot \nabla_x^j \nabla_v^i f_\epsilon  +  \nabla_x^j \nabla_v^i \Gamma(g_\epsilon,f_\epsilon)  \sdot \nabla_x^j \nabla_v^i g_\epsilon \right)   \bdv \bd x \\
& = T_{21} + T_{22} + T_{23}.
\end{aligno}
By the same trick of dealing with $D_1$, we can infer that
\begin{align}
\label{estvt21+t22}
T_{21} +T_{22} \le \epsilon\|(E_\epsilon, B_\epsilon)\|_{H^s_x}\|(f_\epsilon,g_\epsilon)\|_{H^s_\Lambda}^2.
\end{align}
By the assumption on the the quadratic collision operator, we can infer
\begin{aligno}
\label{estvt23}
T_{23} \le \epsilon\|(f_\epsilon, g_\epsilon)\|_{H^s}\|(f_\epsilon, g_\epsilon)\|_{H^s_\Lambda}^2.
\end{aligno}
Combining \eqref{estvt2-0}, \eqref{estvt21+t22} and \eqref{estvt23}, we can infer that
\begin{align}
\label{estvt2}
T_2 \le C\|(f_\epsilon, g_\epsilon, E_\epsilon, B_\epsilon)\|_{H^s_x}\|(f_\epsilon, g_\epsilon)\|_{H^s_\Lambda}^2 + C \sdot \epsilon   \|(\nabla_v f_\epsilon, \nabla_v g_\epsilon)\|_{H^{s-1}}\|(f_\epsilon, g_\epsilon)\|_{H^s_\Lambda}^2.
\end{align}

Combining \eqref{estvleft}, \eqref{estvt1}, \eqref{estvt3} and \eqref{estvt2}, we can infer that
\begin{aligno}
\label{estv-1-0}
& \epsilon^2 \dt \|(\nabla_x^j\nabla_v^i f_\epsilon, \nabla_x^j\nabla_v^i g_\epsilon)\|_{L^2}^2 +  \tfrac{3}{4} \|(\nabla_x^j\nabla_v^i f_\epsilon, \nabla_x^j\nabla_v^i g_\epsilon)\|_{L^2_\Lambda}^2 \\
& \le   4 \|\nabla_v^{i-1} \nabla_x^{j+1} (f_\epsilon, g_\epsilon)\|_{L^2}^2 + 4 \alpha_\epsilon^2\|\nabla_v^i\nabla^j_x (v \sdot E_\epsilon)\|_{L^2}^2 +  C \| (f_\epsilon, g_\epsilon)\|_{H^{k-1}}^2 \\
& +  C \|(f_\epsilon, g_\epsilon, E_\epsilon, B_\epsilon)\|_{H^s_x}\|(f_\epsilon, g_\epsilon)\|_{H^s_\Lambda}^2 +  C \sdot \epsilon   \|(\nabla_v f_\epsilon, \nabla_v g_\epsilon)\|_{H^{s-1}}\|(f_\epsilon, g_\epsilon)\|_{H^s_\Lambda}^2.
\end{aligno}
For $i=1, j = k -1$, \eqref{estv-1-0} becomes
\begin{aligno}
\label{estv-1-1}
& \epsilon^2 \dt \|(\nabla_x^{k-1}\nabla_v^1 f_\epsilon, \nabla_x^{k-1}\nabla_v^1 g_\epsilon)\|_{L^2}^2 +  \tfrac{3}{4} \|(\nabla_x^{k-1}\nabla_v^1 f_\epsilon, \nabla_x^{k-1}\nabla_v^1 g_\epsilon)\|_{L^2_\Lambda}^2 \\
& \le   4 \|  \nabla_x^{k} (f_\epsilon, g_\epsilon)\|_{L^2}^2 + 4 \alpha_\epsilon^2\|\nabla^{k-1}_x E_\epsilon\|_{L^2}^2 +  C \| (f_\epsilon, g_\epsilon)\|_{H^{k-1}}^2 \\
& +  C \|(f_\epsilon, g_\epsilon, E_\epsilon, B_\epsilon)\|_{H^s_x}\|(f_\epsilon, g_\epsilon)\|_{H^s_\Lambda}^2 +  C \sdot \epsilon   \|(\nabla_v f_\epsilon, \nabla_v g_\epsilon)\|_{H^{s-1}}\|(f_\epsilon, g_\epsilon)\|_{H^s_\Lambda}^2.
\end{aligno}
For $i=2, j = k -2$, \eqref{estv-1-0} becomes
\begin{aligno}
\label{estv-1-2}
& \epsilon^2 \dt \|(\nabla_x^{k-2}\nabla_v^2 f_\epsilon, \nabla_x^{k-2}\nabla_v^2 g_\epsilon)\|_{L^2}^2 +  \tfrac{3}{4} \|(\nabla_x^{k-2}\nabla_v^2 f_\epsilon, \nabla_x^{k-2}\nabla_v^2 g_\epsilon)\|_{L^2_\Lambda}^2 \\
& \le   4 \| \nabla_v^1 \nabla_x^{k-1} (f_\epsilon, g_\epsilon)\|_{L^2}^2  +  C \| (f_\epsilon, g_\epsilon)\|_{H^{k-1}}^2 \\
& +  C \|(f_\epsilon, g_\epsilon, E_\epsilon, B_\epsilon)\|_{H^s_x}\|(f_\epsilon, g_\epsilon)\|_{H^s_\Lambda}^2 +  C \sdot \epsilon   \|(\nabla_v f_\epsilon, \nabla_v g_\epsilon)\|_{H^{s-1}}\|(f_\epsilon, g_\epsilon)\|_{H^s_\Lambda}^2.
\end{aligno}
\[\vdots \]
For $i=k, j = 0$, \eqref{estv-1-0} becomes
\begin{aligno}
\label{estv-1-k}
& \epsilon^2 \dt \|( \nabla_v^k f_\epsilon,  \nabla_v^k g_\epsilon)\|_{L^2}^2 +  \tfrac{3}{4} \|( \nabla_v^k f_\epsilon, \nabla_v^k g_\epsilon)\|_{L^2_\Lambda}^2 \\
& \le   4 \| \nabla_v^{k-1} \nabla_x^{1} (f_\epsilon, g_\epsilon)\|_{L^2}^2  +  C \| (f_\epsilon, g_\epsilon)\|_{H^{k-1}}^2 \\
& +  C \|(f_\epsilon, g_\epsilon, E_\epsilon, B_\epsilon)\|_{H^s_x}\|(f_\epsilon, g_\epsilon)\|_{H^s_\Lambda}^2 +  C \sdot \epsilon   \|(\nabla_v f_\epsilon, \nabla_v g_\epsilon)\|_{H^{s-1}}\|(f_\epsilon, g_\epsilon)\|_{H^s_\Lambda}^2.
\end{aligno}
From $8 \times$\eqref{estv-1-1} $+~8\times$\eqref{estv-1-2} $ + \cdots$ $+$ \eqref{estv-1-k}, it follows that there exists $d_1$ such that
\begin{aligno}
\label{estv-k}
& \epsilon^2 \dt \left( 8 \sum\limits_{ i \ge 1, j \ge 1 \atop i + j =k}\|( \nabla_v^i \nabla_x^j f_\epsilon,  \nabla_v^i \nabla_x^j g_\epsilon)\|_{L^2}^2  + \|( \nabla_v^k   f_\epsilon,  \nabla_v^k g_\epsilon)\|_{L^2}^2 \right) +  \tfrac{3}{4} \sum\limits_{ i \ge 1  \atop i + j =k} \|( \nabla_v^i\nabla_x^j f_\epsilon, \nabla_v^i\nabla_x^j g_\epsilon)\|_{L^2_\Lambda}^2\\
& \le 4 \alpha_\epsilon^2 \|E_\epsilon\|_{H^{k-1}_x}^2 + C  \| (f_\epsilon, g_\epsilon)\|_{H^{k-1}_x}^2 + C \|(f_\epsilon, g_\epsilon, E_\epsilon, B_\epsilon)\|_{H^s_x}\|(f_\epsilon, g_\epsilon)\|_{H^s_\Lambda}^2 \\
& +  C \sdot \epsilon   \|(\nabla_v f_\epsilon, \nabla_v g_\epsilon)\|_{H^{s-1}}\|(f_\epsilon, g_\epsilon)\|_{H^s_\Lambda}^2  + d_1  \| (\nabla_v f_\epsilon, \nabla_v g_\epsilon)\|_{H^{k-2}}^2.
\end{aligno}
Denoting
\[ \dot{H}_{v,\epsilon}^{k}(t) =    8 \sum\limits_{ i \ge 1, j \ge 1 \atop i + j =k}\|( \nabla_v^i \nabla_x^j f_\epsilon,  \nabla_v^i \nabla_x^j g_\epsilon)\|_{L^2}^2  + \|( \nabla_v^k   f_\epsilon,  \nabla_v^k g_\epsilon)\|_{L^2}^2,  \]
and choosing $\tfrac{3d_2}{4} \ge d_1 + \tfrac{3}{4}$($d_2 \ge 1$), then we can deduce from \eqref{estv-k}
\begin{aligno}
\label{estv-k-2}
& \epsilon^2 \dt \sum\limits_{k=1}^s  (d_2)^{s-1}\dot{H}_{v,\epsilon}^{k}(t)  +  \tfrac{3}{4} \|(\nabla_v f_\epsilon, \nabla_v g_\epsilon)\|_{H^{s-1}_\Lambda}^2\\
& \le C \|(f_\epsilon, g_\epsilon )\|_{H^{k-1}_x}^2 + C \alpha_\epsilon^2 \|  E_\epsilon\|_{H^{k-1}_x}^2  + C \|(f_\epsilon, g_\epsilon, E_\epsilon, B_\epsilon)\|_{H^s_x}\|(f_\epsilon, g_\epsilon)\|_{H^s_\Lambda}^2 \\
& +  C \sdot \epsilon   \|(\nabla_v f_\epsilon, \nabla_v g_\epsilon)\|_{H^{s-1}}\|(f_\epsilon, g_\epsilon)\|_{H^s_\Lambda}^2,~~k \ge 1.
\end{aligno}
where $\|h\|_{H^0} = \|h\|_{L^2}$.

By the similar method of deducing  \eqref{estv-k}, we can infer that
\begin{aligno}
& \epsilon^2 \dt \left( \tfrac{8c_1}{3} \sum\limits_{m=1}^{k-1} \dot{H}_{v,\epsilon}^{m}(t)  + \dot{H}_{v,\epsilon}^{k}(t)\right)  +  \tfrac{3}{4} \|(\nabla_v f_\epsilon, \nabla_v g_\epsilon)\|_{H^{k-1}_\Lambda}^2\\
& \le C \|(f_\epsilon, g_\epsilon, \alpha_\epsilon E_\epsilon)\|_{H^{s-1}_x}^2   + C \|(f_\epsilon, g_\epsilon, E_\epsilon, B_\epsilon)\|_{H^s_x}\|(f_\epsilon, g_\epsilon)\|_{H^s_\Lambda}^2 \\
& +  C \sdot \epsilon   \|(\nabla_v f_\epsilon, \nabla_v g_\epsilon)\|_{H^{s-1}}\|(f_\epsilon, g_\epsilon)\|_{H^s_\Lambda}^2.
\end{aligno}

\end{proof}

\subsection{The dissipative estimates of the macroscopic parts} From Lemma \ref{lemmav}, if we want to close the whole inequalities, we still need to obtain the dissipative energy estimates of $f_\epsilon,~g_\epsilon$ and $E_\epsilon$. This subsection is devoted to the dissipative energy of $f_\epsilon$ and $g_\epsilon$.
\begin{lemma}
\label{lemmamacrox}
Under the assumption  in the section \ref{sec-assump-on-L} and the assumption \eqref{estMeaninitial} on the initial data, if $(f_\epsilon, g_\epsilon, B_\epsilon, E_\epsilon)$ are solutions to \eqref{vmbtwolinear}, then 
\begin{aligno}
\label{estmacrox}
& \epsilon \sdot \tdt \sum\limits_{k=1}^s \intps\left( \nabla_x \nabla_x^{k-1} f_\epsilon \sdot \nabla_v \nabla_x^{k-1} f_\epsilon + \nabla_x \nabla_x^{k-1} g_\epsilon \sdot \nabla_v \nabla_x^{k-1} g_\epsilon  \right) \bdv \bd x   \\
& + \tfrac{3}{4}\|(\nabla_x f_\epsilon, \nabla_x g_\epsilon)\|_{H^{s-1}_x}^2   +     \| \divg E_\epsilon\|_{H^{s-1}_x}^2  - \delta_2 \|\nabla_v(f_\epsilon, g_\epsilon)\|_{H^{s-1}_{\Lambda_x}}^2 - \delta_1 \|E_\epsilon\|_{H^{s-1}_{ x}}^2  \\
& \le  \left(\tfrac{1}{\delta_1  }  + \tfrac{1}{\delta_2  } \right)  \tfrac{C}{  \epsilon^2}  \| (f_\epsilon
^\perp, g_\epsilon^\perp)\|_{H^s_{\Lambda_x}}^2  +  C \|(f_\epsilon, g_\epsilon, E_\epsilon, B_\epsilon)\|_{H^s_x}\|(f_\epsilon, g_\epsilon)\|_{H^s_\Lambda}^2.
\end{aligno}
\end{lemma}
\begin{proof}
Applying $\nabla_v \nabla_x^{k-1}$ to the first two equations of \eqref{vmbtwolinear}, multiplying them by $\nabla_x \nabla_x^{k-1} f_\epsilon$ and $\nabla_x \nabla_x^{k-1} g_\epsilon$ respectively and then integrating the resulting equation over the phase space, we can infer that
\begin{aligno}
\label{estmacrox-0}
& \epsilon \sdot \tdt \intps\left( \nabla_x \nabla_x^{k-1} f_\epsilon \sdot \nabla_v \nabla_x^{k-1} f_\epsilon + \nabla_x \nabla_x^{k-1} g_\epsilon \sdot \nabla_v \nabla_x^{k-1} g_\epsilon  \right) \bdv \bd x \\
& +  \intps\left( \nabla_v \nabla_x^{k-1} \left( \vdot f_\epsilon\right) \sdot \nabla_x \nabla_x^{k-1} f_\epsilon + \nabla_x \nabla_x^{k-1} \left( \vdot f_\epsilon\right) \sdot \nabla_v \nabla_x^{k-1} f_\epsilon  \right) \bdv \bd x \\
& +  \intps\left( \nabla_v \nabla_x^{k-1} \left( \vdot g_\epsilon\right) \sdot \nabla_x \nabla_x^{k-1} g_\epsilon + \nabla_x \nabla_x^{k-1} \left( \vdot g_\epsilon\right) \sdot \nabla_v \nabla_x^{k-1} g_\epsilon  \right) \bdv \bd x \\
& - \reps \intps\left( \nabla_v \nabla_x^{k-1} \mathcal{L}(f_\epsilon) \sdot \nabla_x \nabla_x^{k-1} f_\epsilon + \nabla_x \nabla_x^{k-1} \mathcal{L}(f_\epsilon) \sdot \nabla_v \nabla_x^{k-1} f_\epsilon  \right) \bdv \bd x \\
& - \reps \intps\left( \nabla_v \nabla_x^{k-1} \mathsf{L}(g_\epsilon) \sdot \nabla_x \nabla_x^{k-1} g_\epsilon + \nabla_x \nabla_x^{k-1} \mathsf{L}(g_\epsilon) \sdot \nabla_v \nabla_x^{k-1} g_\epsilon  \right) \bdv \bd x \\
&  + \repsa \intps \left( \nabla_v \nabla_x^{k-1}(v\!\cdot\!E_\epsilon) \sdot \nabla_x \nabla_x^{k-1} g_\epsilon + \nabla_x \nabla_x^{k-1}(v\!\cdot\!E_\epsilon) \sdot \nabla_v \nabla_x^{k-1} g_\epsilon \bdv \right)\bdv \bd s \\
& = \epsilon \intps \left( \nabla_x \nabla_x^{k-1}N_f \sdot \nabla_v \nabla_x^{k-1} f_\epsilon +  \nabla_v \nabla_x^{k-1}N_f \sdot \nabla_x \nabla_x^{k-1} g_\epsilon \right)  \bdv \bd x\\
& + \epsilon \intps \left( \nabla_x \nabla_x^{k-1}N_g \sdot \nabla_v \nabla_x^{k-1} f_\epsilon +  \nabla_v \nabla_x^{k-1}N_g \sdot \nabla_x \nabla_x^{k-1} g_\epsilon \right)  \bdv \bd x\\
& = M_3 + M_4.
\end{aligno}
We will control the left hand of \eqref{estmacrox-0} respectively. By integration by parts, we can infer
\begin{aligno}
 & \intps\left( \nabla_v \nabla_x^{k-1} \left( \vdot f_\epsilon\right) \sdot \nabla_x \nabla_x^{k-1} f_\epsilon + \nabla_x \nabla_x^{k-1} \left( \vdot f_\epsilon\right) \sdot \nabla_v \nabla_x^{k-1} f_\epsilon  \right) \bdv \bd x \\
 & = \|v \sdot \nabla_x \nabla^{k-1}_x f_\epsilon\|_{L^2} + 2 \intps  \nabla_x \nabla_x^{k-1} \left( \vdot f_\epsilon\right) \sdot \nabla_v \nabla_x^{k-1} f_\epsilon    \bdv \bd x.
\end{aligno}
By the integrating by parts~(three times), we can obtain 
\begin{align}
I & = \intps  \nabla_x \nabla_x^{k-1} \left( \vdot f_\epsilon\right) \sdot \nabla_v \nabla_x^{k-1} f_\epsilon    \bdv \bd x &\\
&  = \intps  (v^j\sdot\nabla_{x_j x_i}^2 \nabla_x^{k-1} f_\epsilon)\partial_{v_i} \nabla_x^{k-1} f_\epsilon \bdv \bd x  \\
&= \|\nabla_x \nabla_x^{k-1}f_\epsilon\|_{L^2}^2 - \|v \sdot \nabla_x \nabla^{k-1}_x f_\epsilon\|_{L^2} - I.
\end{align}
We comments here for more details.  For $k=s$, when all the derivative acts on $f_\epsilon$, the following term
\[ \intps     \left( \vdot  \nabla_x^{s-1} \nabla_x f_\epsilon\right) \sdot \nabla_v \nabla_x^{k-1} f_\epsilon    \bdv \bd x\]
occurs. Formally, this term is   bad. But after integrating by parts three times ($(-1)^3$), there exists a cycle ($I = \text{good terms}~~ - I$). This trick works fine for $M_3$ and $M_4$. 

Thus, we can finally have
\begin{aligno}
\label{estmacro-three-f}
 & \intps\left( \nabla_v \nabla_x^{k-1} \left( \vdot f_\epsilon\right) \sdot \nabla_x \nabla_x^{k-1} f_\epsilon + \nabla_x \nabla_x^{k-1} \left( \vdot f_\epsilon\right) \sdot \nabla_v \nabla_x^{k-1} f_\epsilon  \right) \bdv \bd x \\
 & =\|\nabla_x \nabla_x^{k-1}f_\epsilon\|_{L^2}^2.  
\end{aligno}
For the terms with coefficient $\reps$ in the left hand of \eqref{estmacrox-0}, we denote 
\begin{aligno}
M_1 & = - \reps \intps  \nabla_x \nabla_x^{k-1} \mathcal{L}(f_\epsilon) \sdot \nabla_v \nabla_x^{k-1} f_\epsilon    \bdv \bd x \\
& - \reps \intps  \nabla_v \nabla_x^{k-1} \mathcal{L}(f_\epsilon) \sdot \nabla_x \nabla_x^{k-1} f_\epsilon   \bdv \bd x\\
& = M_{11} + M_{12},
\end{aligno}
and
\begin{aligno}
M_2 & = - \reps \intps  \nabla_x \nabla_x^{k-1} \mathsf{L}(g_\epsilon) \sdot \nabla_v \nabla_x^{k-1} g_\epsilon    \bdv \bd x \\
& - \reps \intps  \nabla_v \nabla_x^{k-1} \mathsf{L}(g_\epsilon) \sdot \nabla_x \nabla_x^{k-1} g_\epsilon   \bdv \bd x.
\end{aligno}
$M_{11}$ is easy  to be controlled. Indeed, 
\begin{align*}
|M_{11}| &  \le   \reps \vert \intps  \nabla_x \nabla_x^{k-1} \mathcal{L}(f_\epsilon) \sdot \nabla_v \nabla_x^{k-1} f_\epsilon   \bdv \bd x \vert \\
& \le \tfrac{1}{2 \delta \epsilon^2} \|\nabla_x^k f_\epsilon^\perp\|_{L^2_\Lambda}^2 + \tfrac{\delta}{2}\|\nabla_v \nabla_x^{k-1} f_\epsilon\|_{L^2_\Lambda}^2.
\end{align*}
For $M_{12}$, by the decomposition of $\nabla_x\nabla_x^{k-1} f_\epsilon$, we can obtain that
\begin{align*}
|M_{12}| &  \le   \reps \vert \intps  \nabla_v \nabla_x^{k-1} \mathcal{L}(f_\epsilon) \sdot \nabla_x \nabla_x^{k-1} f_\epsilon   \bdv \bd x \vert \\
& \le   \reps \vert \intps  \nabla_v \nabla_x^{k-1} \mathcal{L}(f_\epsilon) \sdot \nabla_x \nabla_x^{k-1} \bdp f_\epsilon   \bdv \bd x \vert \\
& + \reps \vert \intps  \nabla_v  \mathcal{L}(\nabla_x^{k-1} f_\epsilon^\perp ) \sdot \nabla_x \nabla_x^{k-1}   f_\epsilon^\perp   \bdv \bd x \vert\\
& = M_{121} + M_{122}.
\end{align*}
For $M_{121}$, by integration by parts, it follows that
\begin{align*}
M_{121} & =   \reps \vert \intps    \mathcal{L}( \nabla_x^{k-1} f_\epsilon^\perp ) \sdot \nabla_x \nabla_x^{k-1} \nabla_v ( \bdp f_\epsilon \m)   \bd v \bd x \vert \\
& \le C  \reps \|\nabla_x^{k-1} f_\epsilon^\perp\|_{L^2_\Lambda} \|\nabla_x \nabla_x^{k-1} \bdp f_\epsilon\|_{L^2} \\
& \le \tfrac{C}{\epsilon^2} \|\nabla_x^{k-1} f_\epsilon^\perp\|_{L^2_\Lambda}^2 + \tfrac{1}{16} \|\nabla_x \nabla_x^{k-1} f_\epsilon\|_{L^2}^2. 
\end{align*}
By the assumptions on the linear Boltzmann operator, we can deduce that for $M_{122}$ 
\begin{align*}
M_{122} & \le C \sdot \reps \left( \| \nabla_x^{k-1} f_\epsilon^\perp\|_{L^2 } +  \|\nabla_v \nabla_x^{k-1} f_\epsilon\|_{L^2_\Lambda} \right)\| \nabla_x \nabla_x^{k-1}   f_\epsilon^\perp\|_{L^2_\Lambda} \\
& \le \tfrac{C}{2\delta_2 \epsilon^2} \|\nabla_x^{k-1}   f_\epsilon^\perp\|_{H^1_{\Lambda_x}}^2  + \tfrac{\delta_2}{2}\|\nabla_v \nabla_x^{k-1} f_\epsilon\|_{L^2_\Lambda}^2.
\end{align*}
With help of the related estimates of $M_1$, we can finally have
\begin{align}
\label{estmacrom1}
|M_1| \le \tfrac{C}{\delta \epsilon^2}  \|\nabla_x^{k-1}   f_\epsilon^\perp\|_{H^1_{\Lambda_x}}^2   + \delta \|\nabla_v \nabla_x^{k-1} f_\epsilon\|_{L^2_\Lambda}^2 + \tfrac{1}{16} \|\nabla_x \nabla_x^{k-1} f_\epsilon\|_{L^2}^2.
\end{align}
Similarly, for $M_2$, we can obtain that
\begin{align}
\label{estmacrom2}
|M_2| \le \tfrac{C}{\delta \epsilon^2}  \|\nabla_x^{k-1}   g_\epsilon^\perp\|_{H^1_{\Lambda_x}}^2   + \delta \|\nabla_v \nabla_x^{k-1} g_\epsilon\|_{L^2_\Lambda}^2 + \tfrac{1}{16} \|\nabla_x \nabla_x^{k-1} g_\epsilon\|_{L^2}^2.
\end{align}
For the last term in the left hand of \eqref{estmacrox-0}, by integration by parts, it follows that
\begin{aligno}
\label{estmacroe}
& \repsa \intps \left( \nabla_v \nabla_x^{k-1}(v\!\cdot\!E_\epsilon) \sdot \nabla_x \nabla_x^{k-1} g_\epsilon + \nabla_x \nabla_x^{k-1}(v\!\cdot\!E_\epsilon) \sdot \nabla_v \nabla_x^{k-1} g_\epsilon \bdv \right)\bdv \bd s \\
& = \repsa \intps \left( \nabla_v \nabla_x^{k-1}(v\!\cdot\!E_\epsilon) \sdot \nabla_x \nabla_x^{k-1} g_\epsilon + \nabla_x \nabla_x^{k-1}(v\!\cdot\!E_\epsilon) \sdot \nabla_v \nabla_x^{k-1} g_\epsilon^\perp \bdv \right)\bdv \bd s \\
& =  \repsa \intps \left(   \nabla_x^{k-1} E_\epsilon \sdot \nabla_x \nabla_x^{k-1} g_\epsilon - \nabla_x^{k-1}(v\!\cdot\!E_\epsilon) \sdot v \sdot \nabla_x \nabla_x^{k-1} g_\epsilon^\perp \bdv \right)\bdv \bd s\\
&  \le -   \|\divg (\nabla_x^{k-1} E_\epsilon)\|_{L^2}^2 + \delta_1 \tfrac{\alpha_\epsilon^2}{\epsilon^2} \|\nabla_x^{k-1} E_\epsilon\|_{L^2}^2 + \tfrac{C}{\delta}\|\nabla_x^k g^\perp_\epsilon\|_{L^2}^2,
\end{aligno}
where we have used the fact that
\[\divg E_\epsilon = \repsa \intv g_\epsilon \bdv,~~\nabla_v g_\epsilon = \nabla_v g_\epsilon^\perp.  \]

In the light of \eqref{estmacrox-0}, \eqref{estmacro-three-f}, \eqref{estmacrom1}, \eqref{estmacrom2} and \eqref{estmacroe}, there exists some $c_2$ such that
\begin{aligno}
\label{estmacrox-1}
& \epsilon \sdot \tdt \intps\left( \nabla_x \nabla_x^{k-1} f_\epsilon \sdot \nabla_v \nabla_x^{k-1} f_\epsilon + \nabla_x \nabla_x^{k-1} g_\epsilon \sdot \nabla_v \nabla_x^{k-1} g_\epsilon  \right) \bdv \bd x   \\
& + \tfrac{3}{4}\|(\nabla_x\nabla_x^{k-1}f_\epsilon, \nabla_x\nabla_x^{k-1}g_\epsilon)\|_{L^2}^2 - \tfrac{C}{\delta \epsilon^2}  \|\nabla_x^{k-1}(f_\epsilon
^\perp, g_\epsilon^\perp)\|_{H^1_{\Lambda_x}}^2 \\
& - c_2 \tfrac{\alpha^2_\epsilon}{\epsilon^2} \|\nabla_x^{k-1} E_\epsilon\|_{L^2}^2  - \delta \|\nabla_v \nabla_x^{k-1}(f_\epsilon, g_\epsilon)\|_{L^2_\Lambda}^2  \\
& \le  M_3 +M_4.
\end{aligno}
With the same trick to that of deducing \eqref{estonlyxD2} and \eqref{estmacro-three-f} (integrating by parts three time), we can deduce that
\begin{align}
M_3 + M_4 \le  C \|(f_\epsilon, g_\epsilon, E_\epsilon, B_\epsilon)\|_{H^s_x}\|(f_\epsilon, g_\epsilon)\|_{H^s_\Lambda}^2.
\end{align}

\end{proof}

\subsection{The dissipative estimates of the electromagnetic parts}

\begin{lemma}
\label{lemma-curl-e}
Under the assumption  in the section \ref{sec-assump-on-L} and the assumption \eqref{estMeaninitial} on the initial data, if $(f_\epsilon, g_\epsilon, B_\epsilon, E_\epsilon)$ are solutions to \eqref{vmbtwolinear}, then  for $k \le s -2$, 
\begin{aligno}
\label{este-dis-e}
&      \tdt \intt \left( |\curl \nabla_x^k E_\epsilon|^2 + |\curl \nabla_x^k B_\epsilon|^2  - 2 \alpha_\epsilon \sdot \curl \nabla_x^k \tilde j_\epsilon \sdot  \curl  \nabla_x^k E_\epsilon \right)\bd x \\
& \qquad +  \tfrac{3\sigma}{4}\tfrac{\alpha_\epsilon^2}{\epsilon^2} \|\curl \nabla_x^k E_\epsilon\|_{L^2}^2  - \tfrac{ \sigma}{16}\tfrac{\alpha_\epsilon^2}{\epsilon^2} \|\curl \nabla_x^k B_\epsilon\|_{L^2}^2\\
 & \le    C \|(f_\epsilon, g_\epsilon, E_\epsilon, B_\epsilon)\|_{H^s_x}^2\|(f_\epsilon, g_\epsilon)\|_{H^s_\Lambda}^2  + C(1+ \tfrac{\epsilon^2}{\gamma_\epsilon^2}) \|\nabla_x^{k} g_\epsilon^\perp\|_{H^2_x}^2.
\end{aligno}
and 
\begin{aligno}
\label{este-dis-e0}
&      \tdt \intt \left( |  E_\epsilon|^2 + |  B_\epsilon|^2  - 2 \alpha_\epsilon \sdot   \tilde j_\epsilon \sdot   E_\epsilon \right)\bd x \\
& \qquad +  \tfrac{3\sigma}{4}\tfrac{\alpha_\epsilon^2}{\epsilon^2} \| E_\epsilon\|_{L^2}^2  - \tfrac{ \sigma}{16}\tfrac{\alpha_\epsilon^2}{\epsilon^2} \|\curl   B_\epsilon\|_{L^2}^2\\
 & \le    C \|(f_\epsilon, g_\epsilon, E_\epsilon, B_\epsilon)\|_{H^s_x}^2\|(f_\epsilon, g_\epsilon)\|_{H^s_\Lambda}^2  + C(1+ \tfrac{\epsilon^2}{\gamma_\epsilon^2}) \|  g_\epsilon^\perp\|_{H^1_x}^2.
\end{aligno}

\end{lemma}

\begin{proof}
Multiplying the second equation of \eqref{vmbtwo-rewrite} by $ \tilde{v} \m$ and then integrating over $\mathbb{R}^3$,
we can infer that
\begin{align}
\label{equationJtilde}
\partial_t \tilde{j}_\epsilon\ + \reps \divg  \intv \tilde{v}\otimes v g_\epsilon  \bdv  -  \sigma \tfrac{\alpha_\epsilon}{\epsilon^2} E_\epsilon - \repst \intv \mathsf{L}(g_\epsilon) \tilde v \bdv  =   \intv \tilde v  N_g \bdv,
\end{align}
with 
\begin{align}
\label{constant-sigma}
\sigma = \tfrac{1}{3} \intv \tilde{v}\sdot v \bdv.
\end{align}
Applying the curl operator to the above equation,  
we finally obtain that
\begin{align}
\label{vmb-curl-e}
\partial_t \nabla\timess \nabla^k_x j_\epsilon + \reps \nabla\timess\left(\divg  \intv \tilde{v}\otimes v \nabla^k_x g_\epsilon  \bdv\right)- \sigma \tfrac{\alpha_\epsilon}{\epsilon^2} \curl \nabla_x^k E_\epsilon + \repst \curl  \nabla_x^k {j}_\epsilon  = \curl\intv \tilde v \nabla_x^k N_g \bdv,
\end{align}
where we have used the fact that
\[   {j}_\epsilon = \intv g_\epsilon \sdot  {v} \bdv = \intv \mathsf{L}(g_\epsilon) \sdot \tilde v \bdv.  \]

Multiplying \eqref{vmb-curl-e} by $-\alpha_\epsilon \curl \nabla_x^k E_\epsilon $, then integrating over torus, we finally deduce that
\begin{aligno}
\label{est-curl-e-1}
& - \alpha_\epsilon \intt \partial_t \curl \nabla_x^k \tilde j_\epsilon \sdot  \curl \nabla_x^k E_\epsilon \bd x -  \tfrac{\alpha_\epsilon}{\epsilon}  \intt  \nabla\timess\left(\divg  \intv \tilde v \otimes v \nabla_x^k g_\epsilon  \bdv\right) \sdot \curl \nabla_x^k E_\epsilon \bd x \\
& +  \sigma \tfrac{  \alpha_\epsilon^2}{\epsilon^2}  \|\curl \nabla_x^k E_\epsilon\|_{L^2}^2 - \tfrac{\alpha_\epsilon    }{\epsilon^2} \intt  \curl \nabla_x^k {j}_\epsilon \sdot \curl \nabla_x^k E_\epsilon \bd x = \alpha_\epsilon \intt \curl\intv \tilde v \nabla_x^k N_g \bdv \sdot \curl \nabla_x^k E_\epsilon \bd x
\end{aligno}
Again, applying the curl operator to the third equation of \eqref{vmbtwolinear}
\[  \gamma_\epsilon\partial_t \curl E_\epsilon - \curl\curl B_\epsilon + \repsb \curl j_\epsilon =0. \]
From the above equation, we can infer that
\begin{aligno}
\label{est-curl-e-2}
& - \alpha_\epsilon \intt  \curl \nabla_x^k \tilde j_\epsilon \sdot \partial_t \curl \nabla_x^k E_\epsilon \bd x +  \tfrac{\alpha_\epsilon}{\gamma_\epsilon}\intt \curl \curl \nabla_x^k B_\epsilon \sdot \curl \nabla_x^k \tilde j_\epsilon \bd x \\
&  - \tfrac{\alpha_\epsilon \beta_\epsilon}{\epsilon \gamma_\epsilon} \intt \curl \nabla_x^k j_\epsilon \sdot \curl \nabla_x^k \tilde{j}_\epsilon \bd x =0.
\end{aligno}
The last term in the right of \eqref{est-curl-e-2} is hard to control. Applying the curl operator to the Ampere's equation and Maxwell's equation in \eqref{vmbtwolinear}, then we can infer that
\begin{align}
\label{estcurle-cancel}
\dt \|\curl \nabla_x^k E_\epsilon, \curl \nabla_x^k B_\epsilon)\|_{L^2}^2 = -\tfrac{\alpha_\epsilon  }{\epsilon^2} \intt  \curl \nabla_x^k {j}_\epsilon \sdot \curl \nabla_x^k E_\epsilon \bd x.
\end{align}
From \eqref{est-curl-e-1}, \eqref{est-curl-e-2} and \eqref{estcurle-cancel},
 \begin{align}
\label{est-curl-3}
\begin{split}
&      \tdt \intt \left( |\curl \nabla_x^k E_\epsilon|^2 + |\curl \nabla_x^k B_\epsilon|^2  - 2 \alpha_\epsilon \sdot \curl \nabla_x^k \tilde j_\epsilon \sdot  \curl  \nabla_x^k E_\epsilon \right)\bd x  +  \sigma\tfrac{\alpha_\epsilon^2}{\epsilon^2} \|\curl \nabla_x^k E_\epsilon\|_{L^2}^2 \\
 & =   \tfrac{\alpha_\epsilon}{\epsilon}\intt  \nabla\timess\left(\divg  \intv \tilde v\otimes v \nabla_x^k g_\epsilon  \bdv\right) \sdot \curl \nabla_x^k E_\epsilon \bd x \\
 &  + \tfrac{\alpha_\epsilon ^2}{\epsilon^2} \intt \curl \nabla_x^k j_\epsilon \sdot \curl \nabla_x^k \tilde{j}_\epsilon \bd x  -     \tfrac{\alpha_\epsilon}{\gamma_\epsilon}\intt \curl \curl \nabla_x^k B_\epsilon \sdot \curl \nabla_x^k \tilde j_\epsilon \bd x\\
 & - \alpha_\epsilon \intt \curl\intv \tilde v \nabla_x^k N_g \bdv \sdot \curl \nabla_x^k E_\epsilon \bd x, 
\end{split}
\end{align}
where we have used the following fact
\[ \alpha_\epsilon \gamma_\epsilon = \epsilon \sdot \beta_\epsilon.   \]

For the first  terms in the right hand of \eqref{est-curl-3} and noticing that $\tilde{v}\otimes v \in \mathrm{Ker}^\perp \mathsf{L}$,
\begin{align*}
  \big\vert \repsa \intt  \nabla\timess\left(\divg  \intv  \tilde{v}\otimes v  \nabla_x^k g_\epsilon  \bdv\right) \sdot \curl \nabla_x^k  E_\epsilon \bd x\big\vert   \le C \|\nabla^{k+2}_x g_\epsilon^\perp\|_{L^2}^2 + \tfrac{\sigma}{16}\tfrac{\alpha_\epsilon^2}{\epsilon^2}\|\curl \nabla_x^k E_\epsilon\|_{L^2}^2.
\end{align*}
For the third term in  \eqref{est-curl-3}, by integration by parts, it follows
\begin{align*}
\tfrac{\alpha_\epsilon}{\gamma_\epsilon}\intt \curl \curl \nabla_x^k B_\epsilon \sdot \curl \nabla_x^k j_\epsilon \bd x = & \intt \curl \curl \nabla_x^k\tilde{j}_\epsilon \sdot \curl \nabla_x^k B_\epsilon \bd x \\
& \le C \tfrac{\epsilon^2}{\gamma_\epsilon^2}\|\nabla^{k+2}_x g_\epsilon^\perp\|_{L^2}^2 + \tfrac{\sigma}{16}\tfrac{\alpha_\epsilon^2}{\epsilon^2}\|\curl \nabla_x^k B_\epsilon\|_{L^2}^2,
\end{align*}
where we have used the following fact that 
\begin{align}
\label{est-nabla-u}
\begin{split}
\|\curl \tilde{j}_\epsilon\|_{L^2}^2 \lesssim \|\curl\intv  g_\epsilon^\perp \tilde{v}   \bdv\|_{L^2}^2 \le C \|\nabla_x g_\epsilon^\perp\|_{L^2}^2.
\end{split}
\end{align}
Similarly, we can infer that
\begin{align}
\tfrac{\alpha_\epsilon ^2}{\epsilon^2} \intt \curl \nabla_x^k j_\epsilon \sdot \curl \nabla_x^k \tilde{j}_\epsilon \bd x \le C \sdot \tfrac{\alpha_\epsilon ^2}{\epsilon^2}\|\nabla_x^{k+1} g_\epsilon^\perp\|_{L^2}^2.
\end{align}
For the last one in the right hand of \eqref{est-curl-3}, noticing that $k \le s-2$, it follows that
\begin{aligno}
\label{esteng0}
& \alpha_\epsilon \intt \curl\intv \tilde v \nabla_x^k N_g \bdv \sdot \curl \nabla_x^k E_\epsilon \bd x \\
& = \alpha_\epsilon \sdot \epsilon \intt\curl \intv \tilde{v}\nabla_x^k\left(   E_\epsilon\sdot v \sdot f_\epsilon   \right) \bdv \sdot \repsa \curl\nabla_x^k E_\epsilon \bd x\\
& -  \intt\curl \intv \tilde{v}\nabla_x^k\left(    ( \alpha_\epsilon E_\epsilon +  \beta_\epsilon v \times B_\epsilon) \sdot \nabla_v f_\epsilon + \Gamma(g_\epsilon, f_\epsilon) \right)\bdv \sdot  \repsa  \curl\nabla_x^k E_\epsilon \bd x \\
& \le C \|(f_\epsilon, g_\epsilon, E_\epsilon, B_\epsilon)\|_{H^s_x}^2\|(f_\epsilon, g_\epsilon)\|_{H^s_\Lambda}^2 +  \tfrac{\sigma}{16}\tfrac{\alpha_\epsilon^2}{\epsilon^2}\|\curl \nabla_x^k E_\epsilon\|_{L^2}^2.
\end{aligno}
By the similar way of deducing \eqref{est-curl-3}, we can infer that
 \begin{align}
\label{est-e-0}
\begin{split}
&      \tdt \intt \left( |  E_\epsilon|^2 + |  B_\epsilon|^2  - 2 \alpha_\epsilon \sdot   \tilde j_\epsilon \sdot    E_\epsilon \right)\bd x  +  \sigma\tfrac{\alpha_\epsilon^2}{\epsilon^2} \|  E_\epsilon\|_{L^2}^2 \\
 & =   \tfrac{\alpha_\epsilon}{\epsilon}\intt  \left(\divg  \intv \tilde v\otimes v   g_\epsilon  \bdv\right) \sdot   E_\epsilon \bd x \\
 &  + \tfrac{\alpha_\epsilon ^2}{\epsilon^2} \intt  j_\epsilon \sdot   \tilde{j}_\epsilon \bd x  -     \tfrac{\alpha_\epsilon}{\gamma_\epsilon}\intt \curl   B_\epsilon \sdot   \tilde j_\epsilon \bd x\\
 & - \alpha_\epsilon \intt  \intv \tilde v   N_g \bdv \sdot   E_\epsilon \bd x, 
\end{split}
\end{align}
and
\begin{aligno}
\label{este-dis-e-0}
&      \tdt \intt \left( |  E_\epsilon|^2 + |  B_\epsilon|^2  - 2 \alpha_\epsilon \sdot   \tilde j_\epsilon \sdot   E_\epsilon \right)\bd x \\
& \qquad +  \tfrac{3\sigma}{4}\tfrac{\alpha_\epsilon^2}{\epsilon^2} \| E_\epsilon\|_{L^2}^2  - \tfrac{ \sigma}{16}\tfrac{\alpha_\epsilon^2}{\epsilon^2} \|\curl   B_\epsilon\|_{L^2}^2\\
 & \le    C \|(f_\epsilon, g_\epsilon, E_\epsilon, B_\epsilon)\|_{H^s_x}^2\|(f_\epsilon, g_\epsilon)\|_{H^s_\Lambda}^2  + C(1+ \tfrac{\epsilon^2}{\gamma_\epsilon^2}) \|  g_\epsilon^\perp\|_{H^1_x}^2.
\end{aligno}
In summary, we complete the proof.

\end{proof}

The next lemma gives the $L^2$ dissipative estimates of $\curl B$.
\begin{lemma}
\label{lemma-curl-b}
Under the assumption  in the section \ref{sec-assump-on-L} and the assumption \eqref{estMeaninitial} on the initial data, if $(f_\epsilon, g_\epsilon, B_\epsilon, E_\epsilon)$ are solutions to \eqref{vmbtwolinear}, then  for $ k \le s -1$.
\begin{align}
\label{est-curl-b}
-\tfrac{\gamma_\epsilon \alpha_\epsilon^2}{\epsilon^2}  \sdot \frac{\bd }{\bd t} \intt \nabla_x^k E_\epsilon\sdot \curl  \nabla_x^k B_\epsilon \bd x + \tfrac{ 3\alpha_\epsilon^2}{4\epsilon^2} \|\curl \nabla_x^k B_\epsilon\|_{L^2}^2 - \tfrac{ \alpha_\epsilon^2}{\epsilon^2}   \|\curl E_\epsilon\|_{L^2}^2   \le C \tfrac{\beta_\epsilon^2}{\epsilon^2} \| \nabla_x^k g_\epsilon^\perp\|_{L^2}^2.
\end{align}

\end{lemma}
\begin{proof}
Applying $\nabla_x^k$ to  the third and forth equations of \eqref{vmbtwolinear}, then multiplying the resulting equation by  $-\curl \nabla_x^k B$  and $-\curl \nabla_x^k E$ respectively, then we can infer that
\begin{align*}
- \tfrac{\gamma_\epsilon \alpha_\epsilon^2}{\epsilon^2}     \sdot \frac{\bd }{\bd t} \intt \nabla_x^k E_\epsilon\sdot \curl \nabla_x^k B_\epsilon\bd x + \tfrac{ \alpha_\epsilon^2}{\epsilon^2} \|\curl \nabla_x^k B_\epsilon\|_{L^2}^2 - \tfrac{  \alpha_\epsilon^2}{\epsilon^2} \|\curl \nabla_x^k E_\epsilon\|_{L^2}^2 = - \tfrac{\beta_\epsilon \alpha_\epsilon^2}{\epsilon^3} \intt \nabla_x^k j_\epsilon \sdot \curl \nabla_x^k B_\epsilon \bd x.
\end{align*}
Recalling that $\repsa \le 1$, then we can infer that
\[ - \tfrac{\beta_\epsilon \alpha_\epsilon^2}{\epsilon^3} \intt \nabla_x^k j_\epsilon \sdot \curl \nabla_x^k B_\epsilon \bd x \le C \tfrac{\beta_\epsilon^2}{\epsilon^2} \| \nabla_x^k g_\epsilon^\perp\|_{L^2}^2 + \tfrac{ \alpha^2_\epsilon }{4 \epsilon^2}\|\curl \nabla_x^k B_\epsilon\|_{L^2}^2.  \]
and
\begin{align*}
-\tfrac{\gamma_\epsilon \alpha_\epsilon^2}{\epsilon^2}  \sdot \frac{\bd }{\bd t} \intt \nabla_x^k E_\epsilon\sdot \curl  \nabla_x^k B_\epsilon \bd x + \tfrac{ 3\alpha_\epsilon^2}{4\epsilon^2} \|\curl \nabla_x^k B_\epsilon\|_{L^2}^2 - \tfrac{ \alpha_\epsilon^2}{\epsilon^2}   \|\curl E_\epsilon\|_{L^2}^2   \le C \tfrac{\beta_\epsilon^2}{\epsilon^2} \| \nabla_x^k g_\epsilon^\perp\|_{L^2}^2,
\end{align*}
where we have used \eqref{est-nabla-u}.
\end{proof}

Combining Lemma \ref{lemma-curl-e}  and Lemma \ref{lemma-curl-b}, we can obtain the following lemma.

\begin{lemma}
\label{lemmacurl-b-e}
Under the assumption  in the section \ref{sec-assump-on-L} and the assumption \eqref{estMeaninitial} on the initial data, if $(f_\epsilon, g_\epsilon, B_\epsilon, E_\epsilon)$ are solutions to \eqref{vmbtwolinear}, then for $k \le s -2$, 
  \begin{align}
  \label{estcurl-b-e}
  \begin{split}
&      \tdt  \sum\limits_{k=0}^{s-2}\intt \left( |\curl \nabla_x^k E_\epsilon|^2 + |\curl \nabla_x^k B_\epsilon|^2  - 2 \alpha_\epsilon \sdot \curl \nabla_x^k \tilde j_\epsilon \sdot  \curl  \nabla_x^k E_\epsilon   -\tfrac{\gamma_\epsilon \alpha_\epsilon^2 \sigma }{4\epsilon^2}  \sdot   \nabla_x^k E_\epsilon\sdot \curl  \nabla_x^k B_\epsilon  \right) \bd x \\
&  +     \tdt \intt \left( |  E_\epsilon|^2 + |  B_\epsilon|^2  - 2 \alpha_\epsilon \sdot   \tilde j_\epsilon \sdot   E_\epsilon \right)\bd x  +  \tfrac{3\sigma}{4}\tfrac{\alpha_\epsilon^2}{\epsilon^2} \| E_\epsilon\|_{L^2}^2  +  \tfrac{ \sigma}{16}\tfrac{\alpha_\epsilon^2}{\epsilon^2} \|(\curl E_\epsilon,\curl   B_\epsilon)\|_{H^{s-2}_x}^2\\
 & \le    C \|(f_\epsilon, g_\epsilon, E_\epsilon, B_\epsilon)\|_{H^s_x}^2\|(f_\epsilon, g_\epsilon)\|_{H^s_\Lambda}^2  + C(1+ \tfrac{\epsilon^2}{\gamma_\epsilon^2}) \|  g_\epsilon^\perp\|_{H^s_x}^2.
  \end{split}
\end{align}
\end{lemma}

\begin{remark}
\label{remark-curl-e-b}
In this lemma, we obtain the dissipative estimates of $E$ and $B$. Noticing that the  Ampere equation and Faraday equation are hyperbolic equations, there is no obvious dissipative effect. To obtain the dissipation of $E$ and $B$, we need the estimates of $\nabla^2_x g$ to close the inequalities.  This is why that  the   second order derivative of initial data belonging to $L^2$ space is necessary.
\end{remark}

\subsection{The whole estimates}
In this section, we shall close the whole estimates. Before that, we first analyze the eletromagnetic parts.

\begin{lemma}
\label{lemmawhole}
Under the assumption  in the section \ref{sec-assump-on-L} and the assumption \eqref{estMeaninitial} on the initial data, if $(f_\epsilon, g_\epsilon, B_\epsilon, E_\epsilon)$ are solutions to \eqref{vmbtwolinear}, then there exists some small enough constant $c_0$ such that 
\begin{align}
\label{estlemmawhole}
\sup\limits_{ 0 \le s \le t}  {H}_\epsilon^s(t) +  \tfrac{1}{4} \int_0^t \left(   \|(   f_\epsilon,    g_\epsilon)\|_{H^{s}_\Lambda}^2 +  \tfrac{\alpha_\epsilon^2}{\epsilon^2} \|(E_\epsilon, B_\epsilon)\|_{H^{s-1}_{ x}}^2 +  \|(   f_\epsilon^\perp,    g_\epsilon^\perp)\|_{H^{s}_{\Lambda_x}}^2 \right)(s)\bd s \le \tfrac {c_u}{c_l}  {H}_\epsilon^s(0),
\end{align}
where $c_l$ and $c_u$ are positive constants only dependent of the Sobolev embedding constant. 
\end{lemma}
\begin{proof}This Lemma can be proved by employing the Poincare's inequality and choosing proper consants.  We split the proof into three steps.

{\bf Step one: obtain the dissipative estimates of eletromagnetic field}

Denoting 
\begin{aligno}
\label{normeletro}
H^{s-1}_{\epsilon,e}(t) &= \sum\limits_{k=0}^{s-2}\intt \left( |\curl \nabla_x^k E_\epsilon|^2 + |\curl \nabla_x^k B_\epsilon|^2  - 2 \alpha_\epsilon \sdot \curl \nabla_x^k \tilde j_\epsilon \sdot  \curl  \nabla_x^k E_\epsilon     \right) \bd x \\
& + \intt \left( |  E_\epsilon|^2 + |  B_\epsilon|^2  - 2 \alpha_\epsilon \sdot   \tilde j_\epsilon \sdot   E_\epsilon \right)\bd x - \sum\limits_{k=0}^{s-2}\intt \left( \tfrac{\gamma_\epsilon \alpha_\epsilon^2 \sigma }{4\epsilon^2}  \sdot   \nabla_x^k E_\epsilon\sdot \curl  \nabla_x^k B_\epsilon  \right) \bd x,
\end{aligno}
and
\begin{aligno}
\label{normmacro}
H^{s}_{d,\epsilon}(t) = \epsilon     \sum\limits_{k=1}^s \intps\left( \nabla_x \nabla_x^{k-1} f_\epsilon \sdot \nabla_v \nabla_x^{k-1} f_\epsilon + \nabla_x \nabla_x^{k-1} g_\epsilon \sdot \nabla_v \nabla_x^{k-1} g_\epsilon  \right) \bdv \bd x,
\end{aligno}
then combining \eqref{estmacrox} and \eqref{estcurl-b-e} up, we can infer that
\begin{aligno}
\label{estmacrox+dbe-0}
& \tdt \bigg( H^{s-1}_{\epsilon,e} + H^{s}_{d,\epsilon}\bigg)(t)  + \tfrac{3}{4}\|(\nabla_x f_\epsilon, \nabla_x g_\epsilon)\|_{H^{s-1}_x}^2          - \delta_2 \|\nabla_v(f_\epsilon, g_\epsilon)\|_{H^{s-1}_{\Lambda_x}}^2 \\
& + \| \divg E_\epsilon\|_{H^{s-1}_x}^2 - \delta_1 \tfrac{\alpha_\epsilon^2}{\epsilon^2} \|E_\epsilon\|_{H^{s-1}_{ x}}^2  +  \tfrac{3\sigma}{4}\tfrac{\alpha_\epsilon^2}{\epsilon^2} \| E_\epsilon\|_{L^2}^2  +  \tfrac{ \sigma}{16}\tfrac{\alpha_\epsilon^2}{\epsilon^2} \|(\curl E_\epsilon,\curl   B_\epsilon)\|_{H^{s-2}_x}^2\\
& \le  C \|(f_\epsilon, g_\epsilon, E_\epsilon, B_\epsilon)\|_{H^s_x}^2\|(f_\epsilon, g_\epsilon)\|_{H^s_\Lambda}^2  + C\bigg(1+ \tfrac{\epsilon^2}{\gamma_\epsilon^2} +  (\tfrac{1}{\delta_1 \epsilon^2  }  + \tfrac{1}{\delta_2  \epsilon^2 }  )    \bigg) \| (f_\epsilon^\perp, g_\epsilon^\perp)\|_{H^s_x}^2.
\end{aligno}
Noticing that  we have obtained the $L^2$ estimates of the curl part and divergence part of $E_\epsilon$,  the estimate of $\nabla E_\epsilon$ can be recovered by Hodge decomposition, that is to say,  
\[    \|\nabla E_\epsilon\|_{L^2}^2 \le C \left( \|\divg E_\epsilon\|_{L^2}^2 +  \|\curl E_\epsilon\|_{L^2}^2 \right).  \]
For the magnetic field $B_\epsilon$,  since $\divg B_\epsilon =0$,
\[    \|\nabla B\|_{L^2}^2 \le C   \|\curl B\|_{L^2}^2.  \]
Furthermore, the forth equation in \eqref{vmbtwolinear},  the mean value of $B$   on the torus are zero. By Poincare's inequality,
\[    \|B_\epsilon\|_{L^2}^2 \le C \|\nabla B_\epsilon\|_{L^2}^2.  \]
So for the electromagnetic field and some $c_4>0$, we can obtain that
\begin{align*}
c_4 \sdot \tfrac{\alpha_\epsilon^2}{\epsilon^2} \|(B_\epsilon, E_\epsilon)\|_{H^{s-1}_x}^2 \le  \| \divg E_\epsilon\|_{H^{s-1}_x}^2    +  \tfrac{3\sigma}{4}\tfrac{\alpha_\epsilon^2}{\epsilon^2} \| E_\epsilon\|_{L^2}^2  +  \tfrac{ \sigma}{16}\tfrac{\alpha_\epsilon^2}{\epsilon^2} \|(\curl E_\epsilon,\curl   B_\epsilon)\|_{H^{s-2}_x}^2.
\end{align*}
Choosing $\delta_1 = \tfrac{c_4}{4}$, then we can infer that
\begin{aligno}
\label{estmacrox+dbe-1}
& \tdt \bigg( H^{s-1}_{\epsilon,e} + H^{s}_{d,\epsilon}\bigg)(t)  + \tfrac{3}{4}\|(\nabla_x f_\epsilon, \nabla_x g_\epsilon)\|_{H^{s-1}_x}^2          - \delta_2 \|\nabla_v(f_\epsilon, g_\epsilon)\|_{H^{s-1}_{\Lambda_x}}^2    + \tfrac{3c_4}{4} \tfrac{\alpha_\epsilon^2}{\epsilon^2} \|(E_\epsilon, B_\epsilon)\|_{H^{s-1}_{ x}}^2   \\
& \le  C \|(f_\epsilon, g_\epsilon, E_\epsilon, B_\epsilon)\|_{H^s_x}^2\|(f_\epsilon, g_\epsilon)\|_{H^s_\Lambda}^2  + C\bigg(1+ \tfrac{\epsilon^2}{\gamma_\epsilon^2}     + \tfrac{1}{\delta_2  \epsilon^2 }       \bigg) \| (f_\epsilon^\perp, g_\epsilon^\perp)\|_{H^s_x}^2.
\end{aligno}

{\bf Step two: obtain the dissipative estimates of the macroscopic part of $f_\epsilon$ and $g_\epsilon$}

By Lemma \ref{lemmamean}, noticing that  $\intt  \bdp g_\epsilon \bd x =0$ for any $t>0$,  thus we can infer that
\begin{align}
\label{estpoing}
\|g_\epsilon\|_{H^s_x}^2 \le C \left( \|\nabla g_\epsilon\|_{H^{s-1}_x}^2 + \|g_\epsilon^\perp\|_{H^s_x}^2\right).
\end{align}
The macroscopic part of $f_\epsilon$ is more complicated. From Lemma \ref{lemmamean},
\begin{align}
\intt \bdp f_\epsilon(t) \bd x = - \gamma_\epsilon  v \sdot \intt E_\epsilon\times B_\epsilon \bd x    - \epsilon \tfrac{|v|^2-3}{6} \|(E_\epsilon, B_\epsilon)\|_{L^2}^2.    
\end{align}
Thus, by Poincare's inequality, we can infer that
\begin{align}
\label{estpoinf}
\|f_\epsilon\|_{H^s_x}^2 \le C \left( \|\nabla f_\epsilon\|_{H^{s-1}_x}^2 + \|f_\epsilon^\perp\|_{H^s_x}^2 + (\gamma_\epsilon^2 + \epsilon^2)\|(E_\epsilon, B_\epsilon)\|_{H^s}^2\|(E_\epsilon, B_\epsilon)\|_{L^2}^2\right).
\end{align}
Together with \eqref{estmacrox+dbe-1}, \eqref{estpoinf} and \eqref{estpoing}, there exists some $c_6>0$ such that
\begin{aligno}
\label{estx+dbe-1}
& \tdt \bigg( H^{s-1}_{\epsilon,e} + H^{s}_{d,\epsilon}\bigg)(t)  + c_6\|(  f_\epsilon,   g_\epsilon)\|_{H^{s}_x}^2          - \delta_2 \|\nabla_v(f_\epsilon, g_\epsilon)\|_{H^{s-1}_{\Lambda_x}}^2    + \tfrac{3c_4}{4} \tfrac{\alpha_\epsilon^2}{\epsilon^2} \|(E_\epsilon, B_\epsilon)\|_{H^{s-1}_{ x}}^2   \\
& \le  C \|(f_\epsilon, g_\epsilon, E_\epsilon, B_\epsilon)\|_{H^s_x}^2\|(f_\epsilon, g_\epsilon)\|_{H^s_\Lambda}^2  + C\bigg(1+ \tfrac{\epsilon^2}{\gamma_\epsilon^2}     + \tfrac{1}{\delta_2  \epsilon^2 }       \bigg) \| (f_\epsilon^\perp, g_\epsilon^\perp)\|_{H^s_x}^2\\
&  + C (\gamma_\epsilon^2 + \epsilon^2)\|(E_\epsilon, B_\epsilon)\|_{H^s}^2\|(E_\epsilon, B_\epsilon)\|_{L^2}^2.
\end{aligno}

{\bf Step three: Closing the entire estimates}

\begin{aligno}
& \epsilon^2 \dt \left( \tfrac{8c_1}{3} \sum\limits_{m=1}^{s-1} \dot{H}_{v,\epsilon}^{m}(t)  + \dot{H}_{v,\epsilon}^{s}(t)\right)  +  \tfrac{3}{4} \|(\nabla_v f_\epsilon, \nabla_v g_\epsilon)\|_{H^{s-1}_\Lambda}^2\\
& \le b_1 \|(f_\epsilon, g_\epsilon)\|_{H^{s-1}_x}^2  +  b_1 \alpha_\epsilon^2\| E_\epsilon\|_{H^{s-1}_x}^2  + C \|(f_\epsilon, g_\epsilon, E_\epsilon, B_\epsilon)\|_{H^s_x}\|(f_\epsilon, g_\epsilon)\|_{H^s_\Lambda}^2 \\
& +  C \sdot \epsilon   \|(\nabla_v f_\epsilon, \nabla_v g_\epsilon)\|_{H^{s-1}}\|(f_\epsilon, g_\epsilon)\|_{H^s_\Lambda}^2,
\end{aligno}
Denoting 
\begin{align}
\label{normv}
 H_{v,\epsilon}^s = \left( \tfrac{8c_1}{3} \sum\limits_{m=1}^{s-1} \dot{H}_{v,\epsilon}^{m}(t)  + \dot{H}_{v,\epsilon}^{s}(t)\right),
\end{align}
  and choosing $b_4$  and $\delta_2$ such that 
\[ b_4 \sdot c_6 \ge b_1 + 1, ~~  \tfrac{3c_4}{4} \sdot b_4 \ge b_1 + \tfrac{1}{2}, ~~b_4\sdot \delta_2 = \tfrac{1}{4}, \]
then we can infer that there exists some $c_8$ such that
\begin{aligno}
\label{estwholelessperpg}
& \tdt \left(  b_4 \sdot H^{s-1}_{\epsilon,e} + b_4 \sdot  H^{s}_{d,\epsilon}  + H_{v,\epsilon}^s   \right) + \tfrac{1}{2} \left( \|(  f_\epsilon,   g_\epsilon)\|_{H^{s}_x}^2           +  \|(  \nabla_v f_\epsilon, \nabla_v  g_\epsilon)\|_{H^{s-1}_\Lambda}^2 +  \tfrac{\alpha_\epsilon^2}{\epsilon^2} \|(E_\epsilon, B_\epsilon)\|_{H^{s-1}_{ x}}^2 \right)  \\
& \le  C \|(f_\epsilon, g_\epsilon, E_\epsilon, B_\epsilon)\|_{H^s_x}^2\|(f_\epsilon, g_\epsilon)\|_{H^s_\Lambda}^2  + c_8\big(1+ \tfrac{\epsilon^2}{\gamma_\epsilon^2}     + \tfrac{1}{\delta_2  \epsilon^2 }       \big) \| (f_\epsilon^\perp, g_\epsilon^\perp)\|_{H^s_x}^2\\
&  + C (\gamma_\epsilon^2 + \epsilon^2)\|(E_\epsilon, B_\epsilon)\|_{H^s}^2\|(E_\epsilon, B_\epsilon)\|_{L^2}^2 +  C \sdot \epsilon   \|(\nabla_v f_\epsilon, \nabla_v g_\epsilon)\|_{H^{s-1}}\|(f_\epsilon, g_\epsilon)\|_{H^s_\Lambda}^2.
\end{aligno}
By the relation \eqref{alphabg}, we can infer that
\[ \gamma_\epsilon \le \tfrac{\alpha_\epsilon}{\epsilon},~~ \epsilon \le \tfrac{\alpha_\epsilon}{\epsilon},~~\tfrac{\epsilon}{\gamma_\epsilon} \le \tfrac{1}{\epsilon}. \]
Denoting
\begin{align}
\label{normtildeH}
\tilde{H}_\epsilon^s := b_5 \|(f_\epsilon, g_\epsilon, E_\epsilon, B_\epsilon)\|_{H^s_x}^2 +  b_4 \sdot H^{s-1}_{\epsilon,e} + b_4 \sdot  H^{s}_{d,\epsilon}  + H_{v,\epsilon}^s,
\end{align}
and choosing $b_5$ such that
\begin{align*}
\tilde{H}_\epsilon^s \approx H_\epsilon^s, ~~\tfrac{b_5}{\epsilon^2} \ge c_8\big(1+ \tfrac{\epsilon^2}{\gamma_\epsilon^2}     + \tfrac{1}{\delta_2  \epsilon^2 }       \big) + \tfrac{1}{2\epsilon^2}, 
\end{align*}
then we have
\begin{aligno}
\label{estwhole-1}
& \tdt \tilde{H}_\epsilon^s + \tfrac{1}{2} \left(   \|(   f_\epsilon,    g_\epsilon)\|_{H^{s}_\Lambda}^2 +  \tfrac{\alpha_\epsilon^2}{\epsilon^2} \|(E_\epsilon, B_\epsilon)\|_{H^{s-1}_{ x}}^2 +  \|(   f_\epsilon^\perp,    g_\epsilon^\perp)\|_{H^{s}_{\Lambda_x}}^2 \right)  \\
& \le  c_9 \tilde{H}_\epsilon^s \left( \|(f_\epsilon, g_\epsilon)\|_{H^s_\Lambda}^2  + \tfrac{\alpha_\epsilon^2}{\epsilon^2} \|(E_\epsilon, B_\epsilon)\|_{H^{s-1}_{ x}}^2 +  \|(   f_\epsilon^\perp,    g_\epsilon^\perp)\|_{H^{s}_{\Lambda_x}}^2\right). 
\end{aligno}
Since $\tilde{H}_\epsilon^s$ is equivalent to $H_\epsilon^s$, there exists some $0<c_l<1$ and $c_u>0$ such that
\begin{align}
\label{estclcu}
c_l \| {H_\epsilon^s} \le \tilde{H}_\epsilon^s \le c_u  H^s_\epsilon.
\end{align}
As long as the initial data satisfy that
\begin{align}
H_\epsilon^s(0) \le c_0:=\tfrac{1}{4 c_u c_9},
\end{align}
then it follow that  for any $ t>0$ 
\begin{aligno}
\sup\limits_{ 0 \le s \le t} \tilde{H}_\epsilon^s(t) +  \tfrac{1}{4} \int_0^t \left(   \|(   f_\epsilon,    g_\epsilon)\|_{H^{s}_\Lambda}^2 +  \tfrac{\alpha_\epsilon^2}{\epsilon^2} \|(E_\epsilon, B_\epsilon)\|_{H^{s-1}_{ x}}^2 +  \repst\|(   f_\epsilon^\perp,    g_\epsilon^\perp)\|_{H^{s}_{\Lambda_x}}^2 \right)(s)\bd s \le \tilde{H}_\epsilon^s(0).
\end{aligno}
By \eqref{estclcu}, we can infer that
\begin{align*}
\sup\limits_{ 0 \le s \le t}  {H}_\epsilon^s(t) +  \tfrac{1}{4} \int_0^t \left(   \|(   f_\epsilon,    g_\epsilon)\|_{H^{s}_\Lambda}^2 +  \tfrac{\alpha_\epsilon^2}{\epsilon^2} \|(E_\epsilon, B_\epsilon)\|_{H^{s-1}_{ x}}^2 + \repst \|(   f_\epsilon^\perp,    g_\epsilon^\perp)\|_{H^{s}_{\Lambda_x}}^2 \right)(s)\bd s \le \tfrac {c_u}{c_l}  {H}_\epsilon^s(0).
\end{align*}
We complete the proof of this lemma. 
\end{proof} 
In what follows, we sketch the idea of constructing approximate solutions and complete the proof of Theorem \ref{theoremexi}.
\begin{align}
\label{vmb-appro}
\begin{cases}
\partial_t f_\epsilon^n +  \reps \vdot f_\epsilon^n  -   \repst \mathcal{L}(f_\epsilon^n)  + \tfrac{\alpha_\epsilon E_\epsilon^{n-1} + \beta_\epsilon v\times B_\epsilon^{n-1}}{\m \epsilon} \sdot \nabla_v (\m g_\epsilon^n) = \reps\Gamma(f_\epsilon^{n-1},f_\epsilon^{n-1}), \\
\partial_t g_\epsilon^n +  \reps \vdot g_\epsilon^n   - \tfrac{\alpha_\epsilon}{\epsilon^2}  E_\epsilon^n \sdot v -   \repst \mathsf{L}(g_\epsilon^n)  +  \tfrac{\alpha_\epsilon E_\epsilon^{n-1}  \beta_\epsilon v\times B_\epsilon^{n-1}}{\m \epsilon} \sdot \nabla_v (\m f_\epsilon^n) = \reps\Gamma(g_\epsilon^{n-1},f_\epsilon^{n-1}), \\
\ges\partial_t E_\epsilon^n - \curl B_\epsilon^n = - \tfrac{\beta_\epsilon}{\epsilon} j_\epsilon^n, \\
\ges  \partial_t B_\epsilon^n + \curl E_\epsilon^n =0,\\
\divg B_\epsilon^n =0,~~\divg E_\epsilon^n = \tfrac{\alpha_\epsilon}{\epsilon} \intv g_\epsilon^n\bdv,
\end{cases}
\end{align}
\begin{align*}
f^0_\epsilon=g^0=0,~~E^0_\epsilon=B^0_\epsilon=0.
\end{align*}
The approximate solutions can be constructed by iteration method and induction method. Then based on the uniform estimates, the solutions can be obtained by employing Rellich-Kondrachov compactness theorem.

\section{Verify the limit}
\label{sec-limit}
In this section, we will complete the proof of Theorem \ref{theoremlimit}  based on the uniform estimates \eqref{estlemmawhole}. 
\begin{proof}[The proof of Theorem \ref{theoremlimit}]
The proof is based on the local conservation laws and uniform estimates. In Sec.\ref{secEstimates}, we have obtain the following prior estimates: for any $t>0$
\begin{align}
\label{estwhole}
\sup\limits_{ 0 \le s \le t}  \|(f_\epsilon,g_\epsilon,B_\epsilon,E_\epsilon)\|_{H^s_x}^2 +   \int_0^t \left(   \|(   f_\epsilon,    g_\epsilon)\|_{H^{s}_\Lambda}^2  +  \repst \|(   f_\epsilon^\perp,    g_\epsilon^\perp)\|_{H^{s}_{\Lambda_x}}^2 \right)(s)\bd s \le C_0.
\end{align}
We split the proof into three steps.  Before that, denoting 
\[ \alpha = \lim\limits_{\epsilon \to 0} \repsa,~~ \beta = \lim\limits_{\epsilon \to 0} \beta_\epsilon,~~\gamma=\lim\limits_{\epsilon \to 0} \gamma_\epsilon.  \]

{\bf Step 1: the strong convergence of $f_\epsilon$ and $g_\epsilon$.}

First, since $\int_0^\infty\|(f_\epsilon,g_\epsilon)\|_{H^s}^2 \bd s \le C_0$, then there exist some $f, g \in L^2((0,\infty);H^s)$ such that
\begin{align}
\label{estconvergencefg}
f_\epsilon \weakc f,~~g_\epsilon \weakc g,~~~~\text{in},~~ L^2\left((0,+\infty); H^{s}\right).
\end{align} 
On the other hand, noticing that there exists a coefficient $\repst$ before the microscopic part in \eqref{estwhole},  then we can infer that
\begin{align}
f^\perp_\epsilon \to 0,~~g^\perp_\epsilon \to 0,~~~~\text{in},~~ L^2\left((0,+\infty); H^s_{ x}\right).
\end{align}   
All together, we can infer 
\[  f_\epsilon \to f,~~g_\epsilon \to g,~~~~\text{in},~~ L^2\left((0,+\infty); H^{s-1}_{x}\right),  \]
and
\[ f \in \mathrm{Ker}\mathcal{L},~~ g \in \mathrm{Ker}\mathsf{L}.  \]
By the properties of the Boltzmann operator, we can infer that there exists some $\rho(t),~u(t),~\theta(t) \in H^s_x$ and $n(t) \in H^s_x$ such that
\begin{align}
\label{estlimitfg}
f(t,x)= \rho(t,x) + u(t,x) \sdot v + \tfrac{|v|^2-3}{2} \theta(t,x),~~g(t,x)=n(t,x).
\end{align}
Recalling the macroscopic parts of $f_\epsilon$ and $g_\epsilon$,  i.e., 
\begin{align*}
\rho_\epsilon = \intv f_\epsilon \bdv,~ u_\epsilon = \intv f_\epsilon v \bdv,~ \theta_\epsilon = \intv h_\epsilon \tfrac{|v|^2 -3}{3} \bdv,~~n_\epsilon = \intv d_\epsilon \bdv.
\end{align*}
By \eqref{estwhole}, we can obtain that
\begin{align}
\label{estLimit3}
\rho_\epsilon(t,x),~~u_\epsilon(t,x),~~\theta_\epsilon(t,x),~~n_\epsilon(t,x) \in L^\infty([0,+\infty); H^s_{x})\cap L^2([0,+\infty); H^s_{x}) . 
\end{align}  
Based on the estimates \eqref{estLimit3}, for any fixed $t>0$, there exist   $\{\rho, ~~u,~~\theta,~~n, B, E\} \subset H^s_x $  such that  for any fixed $T>0$
\begin{align}
\label{estconvergenceMarco}
\rho_\epsilon \to \rho,~~u_\epsilon \to u,~~ \theta_\epsilon \to \theta,~~n_\epsilon \to n,~~B_\epsilon \to B,~~E_\epsilon \to E,~~\text{in},~~ L^2((0,T);H^{s-1}_x).
\end{align}
The next step is to verify $\rho, u, \theta$ and $n$ statisfy the limiting fluid equations.

{\bf Step 2:   the local conservation laws.}

Copying the nonlinear equations \eqref{vmbtwo-rewrite} below,
\begin{align*}
\begin{cases}
\partial_t f_\epsilon +  \reps \vdot f_\epsilon  -   \repst \mathcal{L}(f_\epsilon)  = -  \tfrac{\alpha_\epsilon E_\epsilon + \beta_\epsilon v\times B_\epsilon}{\m \epsilon} \sdot \nabla_v (\m g_\epsilon) + \reps\Gamma(f_\epsilon,f_\epsilon), \\
\partial_t g_\epsilon +  \reps \vdot g_\epsilon   - \tfrac{\alpha_\epsilon}{\epsilon^2}  E_\epsilon \sdot v -   \repst \mathsf{L}(g_\epsilon)  =  -  \tfrac{\alpha_\epsilon E_\epsilon + \beta_\epsilon v\times B_\epsilon}{\m \epsilon} \sdot \nabla_v (\m f_\epsilon) + \reps\Gamma(g_\epsilon,f_\epsilon), \\
\ges\partial_t E_\epsilon - \curl B_\epsilon = - \tfrac{\beta_\epsilon}{\epsilon} j_\epsilon, \\
\ges  \partial_t B_\epsilon + \curl E_\epsilon =0,\\
\divg B_\epsilon =0,~~\divg E_\epsilon = \tfrac{\alpha_\epsilon}{\epsilon} \intv g_\epsilon\bdv.
\end{cases}
\end{align*}
from the above system, the local conservation laws, i.e., the equations of $\rho_\epsilon, u_\epsilon, \theta_\epsilon$ and $n_\epsilon$, are 
\begin{align}
\label{fluid-app}
\begin{cases}
\partial_t \rho_\epsilon + \tfrac{1}{\epsilon} \divg u_\epsilon =0,\\
\partial_t u_\epsilon + \reps \divg \intv \hat{A} \mathcal{L}f_\epsilon \bdv + \reps \nabla_x (\rho_\epsilon + \theta_\epsilon )  = \repsa n_\epsilon \sdot E_\epsilon + \repsb j_\epsilon \timess B_\epsilon,\\
\partial_t \theta_\epsilon + \tfrac{2}{3\epsilon} \divg \intv \hat{B} \mathcal{L}f_\epsilon \bdv + \tfrac{2}{3\epsilon} \divg u_\epsilon  = \tfrac{2}{3} j_\epsilon \cdot E_\epsilon,\\
\partial_t n_\epsilon + \reps \divg j_\epsilon =0,\\
\ges\partial_t E_\epsilon - \curl B_\epsilon = - \tfrac{\beta_\epsilon}{\epsilon} j_\epsilon, \\
\ges  \partial_t B_\epsilon + \curl E_\epsilon =0,\\
\divg B_\epsilon =0,~~\divg E_\epsilon = \tfrac{\alpha_\epsilon}{\epsilon} n_\epsilon.
\end{cases}
\end{align}
where
\begin{aligno}
A(v) = v \otimes v - \tfrac{|v|^2}{3}\mathbb{I},~~B(v) = v ( \tfrac{|v|^2}{2} - \tfrac{5}{2}), ~\mathcal{L}\hat{A}(v) = A(v),~~\mathcal{L}\hat{B}(v) = B(v).
\end{aligno}
By the local conservation laws of mass, i.e.,  the first equation of \eqref{fluid-app},  we can infer that in the distributional sense 
\begin{align}
\label{estConvergenceDivgu}
\divg u_\epsilon \to \divg u  =0.
\end{align}
Before verifying the the velocity and temperature equation, we need to understand the limiting behavior of $j_\epsilon$. 

{\bf Step 3, the limiting behavior of $ \reps j_\epsilon$.}

From \eqref{equationJtilde}, we can infer that
\begin{align*}
\reps j_\epsilon =  -\epsilon \partial_t \tilde{j}_\epsilon\ -  \divg  \intv \tilde{v}\otimes v g_\epsilon  \bdv  +  \sigma \tfrac{\alpha_\epsilon}{\epsilon} E_\epsilon + \epsilon   \intv \tilde v  N_g \bdv .
\end{align*}
By the uniform estimates \eqref{estwhole}, for any $t>0$, 
\begin{align}
\label{estlimitj}
\int_0^t \|  ( j_\epsilon, \tilde{j}_\epsilon)(s)\|_{L^2}^2 \bd s \le  C_0 \epsilon^2, 
\end{align}
and
\begin{align}
 \divg  \intv \tilde{v}\otimes v g_\epsilon  \bdv = \divg  \intv \tilde{v}\otimes v g_\epsilon^\perp  \bdv +\sigma\nabla_x n_\epsilon= \sigma\nabla_x n_\epsilon + R_1(\epsilon).
\end{align}
where 
\[ R_1(\epsilon) \to 0,~~\text{in},~~ L^2((0,+\infty);H^{s-1}_x).  \]
For the last term in the right hand of $\reps j_\epsilon$,  by simple computation and decomposition, we can infer that
\begin{aligno}
\label{estlimitje}
\epsilon   \intv \tilde v  N_g \bdv & =    - \intv \bigg(\left(\alpha_\epsilon E_\epsilon + \beta_\epsilon v\times B_\epsilon\right)  \sdot \nabla_v (\m f_\epsilon)\bigg) \tilde{v} \bd v +  \intv  \Gamma(g_\epsilon,f_\epsilon)\tilde{v}\bdv\\
& =- \beta_\epsilon \intv  \left(   v\times B_\epsilon   \sdot  \nabla_v f_\epsilon  \right)  \tilde{v} \bdv + n_\epsilon \intv \Gamma(1,\bdp f_\epsilon) \tilde{v}\bdv + R_2(\epsilon),
\end{aligno}
where 
\[ R_2(\epsilon) =  \alpha_\epsilon\intv   \left( E_\epsilon    \sdot \nabla_v (\m f_\epsilon) \right) \tilde{v} \bd v  +   \intv  \Gamma(g_\epsilon^\perp,f_\epsilon)\tilde{v}\bdv +  \intv  \Gamma(g_\epsilon,f_\epsilon^\perp)\tilde{v}\bdv.  \]
According to \eqref{estconvergencefg}, \eqref{estconvergenceMarco} and \eqref{estlimitfg},  we can infer that in the distributional sence:
\begin{align}
\intv  \left(   v\times B_\epsilon   \sdot  \nabla_v f_\epsilon  \right)  \tilde{v} \bdv \to \sigma B\timess u.
\end{align}
Noticing the fact  that $ \bdp f_\epsilon = \rho_\epsilon + u_\epsilon\sdot v + \tfrac{|v|^2-3}{3} \theta_\epsilon$ and $\bdp f_\epsilon \in \mathrm{Ker}\mathcal{L}$, then we have
\begin{align*}
\mathcal{L}(u_\epsilon\sdot v + \tfrac{|v|^2-3}{3} \theta_\epsilon) = \Gamma(u_\epsilon\sdot v + \tfrac{|v|^2-3}{3} \theta_\epsilon,1) + \Gamma(1,u_\epsilon\sdot v + \tfrac{|v|^2-3}{3} \theta_\epsilon)=0.
\end{align*}
Then, for the rest one in the second line of \eqref{estlimitje}, we can deduce that
\begin{aligno}
\label{estnu}
n_\epsilon \intv \tilde{v} \Gamma(  1,  \bdp f_\epsilon) \bdv & = n_\epsilon \intv \tilde{v} \Gamma(  1,  u_\epsilon\sdot v + \tfrac{|v|^2-3}{3} \theta_\epsilon) \bdv \\
& = - n_\epsilon \intv \tilde{v} \Gamma(u_\epsilon\sdot v + \tfrac{|v|^2-3}{3} \theta_\epsilon, 1) \bdv \\
& = n_\epsilon \intv \tilde{v} \mathsf{L}(u_\epsilon\sdot v + \tfrac{|v|^2}{3} \theta_\epsilon) \bdv \\
& =  n_\epsilon \intv  {v} (u_\epsilon\sdot v + \tfrac{|v|^2-3}{3} \theta_\epsilon ) \bdv \\
& = n_\epsilon u_\epsilon.
\end{aligno}
All together, we can infer that in the distributional sense:
\begin{align}
\epsilon   \intv \tilde v  N_g \bdv    \to  \beta  u\timess B + n u.
\end{align}
and
\begin{align}
\label{estConvergeje}
\reps j_\epsilon \to j= \sigma (  \alpha E  + \beta u\timess B) -  \sigma\nabla_x n + n u.
\end{align}

{\bf Step 4, the convergence of conservation laws of momentum and temperature}

From \eqref{estConvergeje}, we can infer that in the distributional sense:
\begin{align}
\label{limitj}
j_\epsilon \to 0,~~\reps j_\epsilon \to j= \tfrac{\sigma}{3} (  \alpha E  + \beta u\timess B) -  \tfrac{\sigma}{3}\nabla_x n + n u. 
\end{align}

Furthermore, for   velocity and temperature equation of \eqref{fluid-app}, we can infer that 
\begin{aligno}
\label{nsp-app}
\partial_t u_\epsilon + \reps \divg \intv \hat{A} \mathcal{L}f_\epsilon \bdv + \reps \nabla_x (\rho_\epsilon + \theta_\epsilon ) & =  \repsa n_\epsilon \sdot E_\epsilon + \repsb j_\epsilon \timess B_\epsilon, \\
\partial_t \left(\tfrac{3}{5}\theta_\epsilon - \tfrac{2}{5}\rho_\epsilon \right) + \tfrac{2}{5\epsilon} \divg \intv \hat{B} \mathcal{L}f_\epsilon \bdv   & = \tfrac{2}{5} j_\epsilon \sdot E_\epsilon. 
\end{aligno}
{Based on \eqref{estwhole} and \eqref{estlimitj},    it follows that}  
\[ \nabla_x (\rho_\epsilon + \theta_\epsilon  ) \to 0, ~~\text{ in the distribution sense, } \]
and
\begin{align}
\label{estConvergencerhotheta}
  \rho + \theta  = 0. 
\end{align}
The  \eqref{estConvergenceDivgu} and \eqref{estConvergencerhotheta} are Bousinessq relations. Now, we can deduce the velocity equations of the macroscopic equations. Denoting by $\mathbf{P}$  the Leray projection operator on torus, from the local conservation laws \eqref{fluid-app},  it follows that 
\begin{aligno}
\label{nsp-app-4}
\partial_t \mathbf{P} u_\epsilon + \reps \mathbf{P}\left( \divg \intv \hat{A} \mathcal{L}f_\epsilon \bdv \right)   & = \mathbf{P} \left( \repsa n_\epsilon E_\epsilon + \repsb j_\epsilon\timess B_\epsilon \right). 
\end{aligno}
Based on the first equation of \eqref{vmbtwo-rewrite}, it implies
\begin{align*}
     \reps \mathcal{L}(f_\epsilon)     & = -\vdot f_\epsilon - \epsilon \partial_t f_\epsilon +  \Gamma(f_\epsilon, f_\epsilon) +   \alpha_\epsilon  E_\epsilon\sdot v \sdot g_\epsilon -  ( \alpha_\epsilon E_\epsilon + \beta_\epsilon v \times B_\epsilon) \sdot \nabla_v g_\epsilon    \\
     & =    {\Gamma}(f_\epsilon,f_\epsilon) - \vdot   f_\epsilon  + R_{\mathcal{L}}(\epsilon).
\end{align*}
By simple calculation (see \cite{diogosrm-2019-vmb-fluid,bgl1993convergence,bgl1991formal}),  we can infer that 
\begin{align}
\label{aaa}
  \intv \hat{A} \sdot \reps \mathcal{L}h_\epsilon \bdv = u_\epsilon \otimes u_\epsilon - \tfrac{|u_\epsilon|^2}{3} \mathbf{I} - \mu \left(\nabla_x u_\epsilon + \nabla^T_\epsilon u_\epsilon -\tfrac{2}{3} \divg u_\epsilon \mathbf{I}\right)- R_f(\epsilon)
\end{align}
with
\[R_f(\epsilon) :=  \intv \hat A \sdot \left( R_{\mathcal{L}}(\epsilon) -  \vdot f^\perp_\epsilon +  {\Gamma}(f_\epsilon^\perp,f_\epsilon) + {\Gamma}(f_\epsilon,f_\epsilon^\perp) \right) \bdv,  \]
and
\[   \mu = \tfrac{1}{15} \sum\limits_{ 1 \le i \le 3 \atop 1 \le j \le 3}\intv A_{ij}\hat{A}_{ij}\bdv. \]
According to \eqref{estwhole},  by the weak convergence results \eqref{estconvergencefg} and the properties   \eqref{estlimitfg} of $f$ and $g$, it holds in the distributional sense that 
\[ R_f(\epsilon) \to 0. \]
Based on \eqref{estwhole},  \eqref{estConvergeje},  \eqref{estLimit3} and \eqref{nsp-app-4},  we can obtain that
\begin{align}
\partial_t \mathbf{P} \mathbf{u}_\epsilon  \in H^{s-1}_{x}.
\end{align}
Employing  Aubin-Lions-Simon theorem (see \cite{tools-ns}), we can infer 
\begin{align}
\label{strongu}
\mathbf{P} \mathbf{u}_\epsilon  \in C((0,+\infty;H^{s-1}_{x});~~ \mathbf{P} \mathbf{u}_\epsilon \to    \mathbf{u},~~  \text{in}~~ C((0,+\infty;H^{s-1}_{x}). 
\end{align}

Thus, according to \eqref{estwhole}, \eqref{estLimit3} and \eqref{aaa},  we can infer in the distributional sense that
\begin{aligno}
& \mathbf{P} u_\epsilon \to u,~~\reps \mathbf{P}\left( \divg \intv \hat{A} \mathcal{L}h_\epsilon \bdv \right)  \to u \sdot \nabla u - \mu \Delta u.
\end{aligno}
Combining the strong convergence properties \eqref{estconvergenceMarco}, we can finally deduce that
\begin{align}
\label{fluidlimitu}
\partial_t u + u\sdot\nabla u - \mu\Delta u + \nabla P = \alpha n \sdot E + \beta \sdot j \timess B.
\end{align}
{Similar to \eqref{aaa}, we can infer that
\begin{align}
\label{bbb}
  \tfrac{2}{5}   \intv \hat{B} \sdot \reps \mathcal{L}g_\epsilon \bdv = u_\epsilon\sdot \theta_\epsilon - \kappa \nabla \theta_\epsilon -  R_\theta(\epsilon)
\end{align}
with
\[R_\theta(\epsilon) :=  \tfrac{5}{2}\intv \hat B \sdot \left( R_\mathcal{L}(\epsilon) -  \vdot f^\perp_\epsilon +  {\Gamma}(f_\epsilon^\perp,f_\epsilon) + {\Gamma}(f_\epsilon,f_\epsilon^\perp) \right) \bdv,~~\kappa = \tfrac{2}{15} \sum\limits_{ 1 \le i \le 3  }\intv B_{i}\hat{B}_{i}\bdv. \]}
Finally, the temperature equation becomes
\begin{align}
\label{fluid-tem}
\partial_t \left(\tfrac{3}{5}\theta_\epsilon - \tfrac{2}{5}\rho_\epsilon \right)  + \divg(u_\epsilon \theta)- \kappa \Delta \theta_\epsilon   & = \tfrac{2}{5} j_\epsilon \sdot E_\epsilon + \divg R_\theta(\epsilon).
\end{align}
According to \eqref{estconvergenceMarco}, \eqref{estConvergencerhotheta} and \eqref{limitj},   by the similar analysis to obtaining \eqref{fluidlimitu},  we can infer that
\begin{align}
\partial_t \theta + u \sdot\nabla \theta - \kappa \Delta \theta =0.
\end{align}
In the light of all the previous analysis in the section, it follows that we have verified that 
\begin{align}
\label{limit-g}
f_\epsilon(t,x,v) \weakc \rho(t,x) + u(t,x) \sdot v + \tfrac{|v|^2-3}{2} \theta(t,x), ~~ g_\epsilon(t,x,v) \weakc n(t,x), ~~~~\text{in},~~ L^2\left((0,+\infty); H^{s}\right).
\end{align} 
with $\rho,~u,~\theta \in  L^\infty([0,+\infty);H^s_{x})$  and satisfying 
\begin{align}
\label{nsfp-c}
\begin{cases}
\partial_t u + u\sdot \nabla u - \nu \Delta u + \nabla P = \alpha n \sdot E + \beta \sdot j \timess B,\\
\partial_t \theta + u \sdot\nabla \theta - \kappa \Delta \theta =0,\\
\divg u = 0,~~ \rho + \theta =0,\\
j= \sigma (  \alpha E  + \beta u\timess B) -  \sigma\nabla_x n + n u
\end{cases}
\end{align}

{\bf Step 5, the electromagnetic system}

For the Ampere's equation and Faraday's equation in \eqref{fluid-app}, according to the uniform estimates \eqref{estwhole} and \eqref{limitj}, we can infer that in the distributional sense:
\begin{align}
\label{electro}
\begin{cases}
\gamma \partial_t E - \curl B = - \beta j,\\
\gamma \partial_t B + \curl E =0,\\
\divg B =0, ~~\divg E = \alpha n.
\end{cases}
\end{align}
Furthermore, from the forth equation of \eqref{fluid-app}, we can infer that
\begin{align}
\label{equationn}
\partial_t n + \divg j =0.
\end{align}

{\bf  Step 6, The final limiting fluid equation}

Combining \eqref{nsfp-c} and \eqref{electro},  the limiting system is
\begin{align}
\label{nsfwpf}
\begin{cases}
\partial_t u + u\sdot \nabla u - \nu \Delta u + \nabla P = \alpha n \sdot E + \beta \sdot j \timess B,\\
\partial_t \theta + u \sdot\nabla \theta - \kappa \Delta \theta =0,\\
\divg u = 0,~~ \rho + \theta =0,\\
j= \sigma (  \alpha E  + \beta u\timess B) - \sigma\nabla_x n + n u,\\
\gamma \partial_t E - \curl B = - \beta j,\\
\gamma \partial_t B + \curl E =0,\\
\divg B =0, ~~\divg E = \alpha n.
\end{cases}
\end{align}

Then for the limiting fluid equations, we only need to plug the value of $\alpha,~\beta,~\gamma$ into \eqref{nsfwpf}.

For \eqref{relationnsw}, we have
\begin{align}
\alpha = \beta =\gamma =1,
\end{align}
and
\begin{align*}
\begin{cases}
\partial_t u + u\sdot \nabla u - \nu \Delta u + \nabla P =   n \sdot E +     j \timess B,\\
\partial_t \theta + u \sdot\nabla \theta - \kappa \Delta \theta =0,\\
\divg u = 0,~~ \rho + \theta =0,\\
\partial_t E - \curl B = -   j,\\
\partial_t B + \curl E =0,\\
\divg B =0, ~~\divg E =  n,\\
j= \sigma (   E  +  u\timess B) -  \sigma\nabla_x n + n u
\end{cases}
\end{align*}

As for \eqref{relationnsp}, it follows that
\begin{align}
\alpha =1,  \beta =\gamma =0,
\end{align}
and
\begin{align*}
\begin{cases}
\partial_t u + u\sdot \nabla u - \nu \Delta u + \nabla P =   n \sdot E,\\
\partial_t \theta + u \sdot\nabla \theta - \kappa \Delta \theta =0,\\
\partial_t n + \divg(nu) -  {\sigma} \Delta n +  {\sigma} n =0,\\
\divg u = 0,~~ \rho + \theta =0,\\
\curl E =0, ~~\divg E =  n,
\end{cases}
\end{align*}
where we have used \eqref{equationn}.

With respect to \eqref{relationnsp}, it follows that
\begin{align}
\alpha =\beta =\gamma =0,
\end{align}
and two fluids Navier-Stokes-Fourier system, 
\begin{align*}
\begin{cases}
\partial_t u + u\sdot \nabla u - \nu \Delta u + \nabla P = 0,\\
\partial_t \theta + u \sdot\nabla \theta - \kappa \Delta \theta =0,\\
\partial_t n + u\sdot\nabla n - \sigma\Delta n  =0,\\
\divg u = 0,~~ \rho + \theta =0.
\end{cases}
\end{align*}

\end{proof}


\begin{thebibliography}{10}

\bibitem{diogosrm-2019-vmb-fluid}
D.~Ars\'{e}nio and L.~Saint-Raymond.
\newblock \emph{From the {V}lasov-{M}axwell-{B}oltzmann system to
  incompressible viscous electro-magneto-hydrodynamics. {V}ol. 1}.
\newblock EMS Monographs in Mathematics. European Mathematical Society (EMS),
  Z\"{u}rich (2019).

\bibitem{bgl1991formal}
C.~Bardos, F.~Golse and C.~D. Levermore,
\newblock Fluid dynamic limits of kinetic equations. {I}. {F}ormal derivations.
\newblock \emph{J. Statist. Phys.} \textbf{63} (1991), 323--344.

\bibitem{bgl1993convergence}
C.~Bardos, F.~Golse and C.~D. Levermore,
\newblock Fluid dynamic limits of kinetic equations. {II}. {C}onvergence proofs
  for the {B}oltzmann equation.
\newblock \emph{Comm. Pure Appl. Math.} \textbf{46} (1993), 667--753.

\bibitem{bu1991}
C.~BARDOS and S.~UKAI,
\newblock The classical incompressible navier-stokes limit of the boltzmann
  equation.
\newblock \emph{Mathematical Models and Methods in Applied Sciences}
  \textbf{01} (1991), 235--257.

\bibitem{belm-2000-binary}
S.~Bastea, R.~Esposito, J.~L. Lebowitz and R.~Marra,
\newblock Binary fluids with long range segregating interaction. {I}.
  {D}erivation of kinetic and hydrodynamic equations.
\newblock \emph{J. Statist. Phys.} \textbf{101} (2000), 1087--1136.

\bibitem{tools-ns}
F.~Boyer and P.~Fabrie.
\newblock \emph{Mathematical tools for the study of the incompressible
  {N}avier-{S}tokes equations and related models}, \emph{Applied Mathematical
  Sciences}, vol. 183.
\newblock Springer, New York (2013).

\bibitem{briant-2015-be-to-ns}
M.~Briant,
\newblock From the {B}oltzmann equation to the incompressible {N}avier-{S}tokes
  equations on the torus: a quantitative error estimate.
\newblock \emph{J. Differential Equations} \textbf{259} (2015), 6072--6141.

\bibitem{diperna-lions1989cauchy}
R.~J. DiPerna and P.-L. Lions,
\newblock On the {C}auchy problem for {B}oltzmann equations: global existence
  and weak stability.
\newblock \emph{Ann. of Math. (2)} \textbf{130} (1989), 321--366.

\bibitem{dlyz2017cmp}
R.~Duan, Y.~Lei, T.~Yang and H.~Zhao,
\newblock The {V}lasov-{M}axwell-{B}oltzmann system near {M}axwellians in the
  whole space with very soft potentials.
\newblock \emph{Comm. Math. Phys.} \textbf{351} (2017), 95--153.

\bibitem{gsrm2004}
F.~Golse and L.~Saint-Raymond,
\newblock The {N}avier-{S}tokes limit of the {B}oltzmann equation for bounded
  collision kernels.
\newblock \emph{Invent. Math.} \textbf{155} (2004), 81--161.

\bibitem{gsrm-2005-survery}
F.~Golse and L.~Saint-Raymond,
\newblock Hydrodynamic limits for the {B}oltzmann equation.
\newblock \emph{Riv. Mat. Univ. Parma (7)} \textbf{4**} (2005), 1--144.

\bibitem{uvpb2020}
M.~Guo, N.~Jiang and Y.-L. Luo,
\newblock From vlasov-poisson-boltzmann system to incompressible
  navier-stokes-fourier-poisson system: convergence for classical solutions.
\newblock \emph{arXiv preprint arXiv:2006.16514}  (2020).

\bibitem{guo-2003-vmb-invention}
Y.~Guo,
\newblock The {V}lasov-{M}axwell-{B}oltzmann system near {M}axwellians.
\newblock \emph{Invent. Math.} \textbf{153} (2003), 593--630.

\bibitem{guo2006NSlimit}
Y.~Guo,
\newblock Boltzmann diffusive limit beyond the {N}avier-{S}tokes approximation.
\newblock \emph{Comm. Pure Appl. Math.} \textbf{59} (2006), 626--687.

\bibitem{vmbtonsp}
J.~Jang,
\newblock Vlasov-{M}axwell-{B}oltzmann diffusive limit.
\newblock \emph{Arch. Ration. Mech. Anal.} \textbf{194} (2009), 531--584.

\bibitem{jama2012siam}
J.~Jang and N.~Masmoudi,
\newblock Derivation of {O}hm's law from the kinetic equations.
\newblock \emph{SIAM J. Math. Anal.} \textbf{44} (2012), 3649--3669.

\bibitem{vmbtonswu}
N.~Jiang and Y.-L. Luo,
\newblock From {V}lasov-{M}axwell-{B}oltzmann system to two-fluid
  incompressible {N}avier-{S}tokes-{F}ourier-{M}axwell system with ohm's law:
  convergence for classical solutions.
\newblock \emph{arXiv preprint arXiv:1905.04739}  (2019).

\bibitem{vmbtonswh}
N.~Jiang, Y.-L. Luo and T.-F. Zhang,
\newblock Incompressible navier-stokes-fourier-maxwell system with ohm's law
  limit from vlasov-maxwell-boltzmann system: Hilbert expansion approach.
\newblock \emph{arXiv preprint arXiv:2007.02286}  (2020).

\bibitem{ns-limit-2018}
N.~Jiang, C.-J. Xu and H.~Zhao,
\newblock Incompressible {N}avier-{S}tokes-{F}ourier limit from the {B}oltzmann
  equation: classical solutions.
\newblock \emph{Indiana Univ. Math. J.} \textbf{67} (2018), 1817--1855.

\bibitem{jz2020vpbconvergence}
N.~Jiang and X.~Zhang,
\newblock Sensitivity analysis and incompressible {N}avier-{S}tokes-{P}oisson
  limit of {V}lasov-{P}oisson-{B}oltzmann equations with uncertainty.
\newblock \emph{arXiv preprint arXiv:2007.00879}  (2020).

\bibitem{lm2010soft}
C.~D. Levermore and N.~Masmoudi,
\newblock From the {B}oltzmann equation to an incompressible
  {N}avier-{S}tokes-{F}ourier system.
\newblock \emph{Arch. Ration. Mech. Anal.} \textbf{196} (2010), 753--809.

\bibitem{vpb2020limit-spectrum}
H.-L. Li, T.~Yang and M.~Zhong,
\newblock Diffusion limit of the vlasov-poisson-boltzmann system.
\newblock \emph{arXiv preprint arXiv:2007.01461}  (2020).

\bibitem{lm2001acoustic}
P.-L. Lions and N.~Masmoudi,
\newblock From the {B}oltzmann equations to the equations of incompressible
  fluid mechanics. {I}, {II}.
\newblock \emph{Arch. Ration. Mech. Anal.} \textbf{158} (2001), 173--193,
  195--211.

\bibitem{masmoudi-srm2003stokesfourier}
N.~Masmoudi and L.~Saint-Raymond,
\newblock From the {B}oltzmann equation to the {S}tokes-{F}ourier system in a
  bounded domain.
\newblock \emph{Comm. Pure Appl. Math.} \textbf{56} (2003), 1263--1293.

\bibitem{mischler2010asens}
S.~Mischler,
\newblock Kinetic equations with {M}axwell boundary conditions.
\newblock \emph{Ann. Sci. \'Ec. Norm. Sup\'er. (4)} \textbf{43} (2010),
  719--760.

\bibitem{mouhot-2006-cpde}
C.~Mouhot,
\newblock Explicit coercivity estimates for the linearized {B}oltzmann and
  {L}andau operators.
\newblock \emph{Comm. Partial Differential Equations} \textbf{31} (2006),
  1321--1348.

\bibitem{mouhot-2006-homogeneous}
C.~Mouhot,
\newblock Rate of convergence to equilibrium for the spatially homogeneous
  {B}oltzmann equation with hard potentials.
\newblock \emph{Comm. Math. Phys.} \textbf{261} (2006), 629--672.

\bibitem{mouhotneumann-2006-decay}
C.~Mouhot and L.~Neumann,
\newblock Quantitative perturbative study of convergence to equilibrium for
  collisional kinetic models in the torus.
\newblock \emph{Nonlinearity} \textbf{19} (2006), 969--998.

\bibitem{sr2006-vmb}
R.~M. Strain,
\newblock The {V}lasov-{M}axwell-{B}oltzmann system in the whole space.
\newblock \emph{Comm. Math. Phys.} \textbf{268} (2006), 543--567.

\bibitem{twovpblimits}
Y.~Wang,
\newblock The diffusive limit of the vlasov–boltzmann system for binary
  fluids.
\newblock \emph{SIAM Journal on Mathematical Analysis} \textbf{43} (2011),
  253--301.

\end{thebibliography}
\end{document}